\documentclass{amsart}
\usepackage{amsfonts,amssymb,amsmath,amsthm,amscd}
\usepackage{url}
\usepackage{enumerate}

\newtheorem{theorem}{Theorem}[section]
\newtheorem{lemma}[theorem]{Lemma}
\newtheorem{proposition}[theorem]{Proposition}
\newtheorem{corollary}[theorem]{Corollary}

\newtheorem{hypothesis}[theorem]{Hypothesis}
\theoremstyle{definition}

\newenvironment{definition}[1][Definition]{\begin{trivlist}
\item[\hskip \labelsep {\bfseries #1}]}{\end{trivlist}}
\newenvironment{remark}[1][Remark]{\begin{trivlist}
\item[\hskip \labelsep {\bfseries #1}]}{\end{trivlist}}

\newcommand{\ab}{{\operatorname{ab}}}

\newcommand{\Ram}{{\operatorname{Ram}}}

\newcommand{\rec}{{\operatorname{rec}}}
\newcommand{\nrd}{{\operatorname{nrd}}}

\newcommand{\End}{{\operatorname{End}}}

\newcommand{\Hom}{{\operatorname{Hom}}}

\newcommand{\Aut}{{\operatorname{Aut}}}

\newcommand{\Tr}{{\operatorname{Tr}}}
\newcommand{\Stab}{{\operatorname{Stab}}}

\newcommand{\Pic}{{\operatorname{Pic}}}

\newcommand{\new}{{\operatorname{new}}}
\newcommand{\pr}{{\operatorname{pr}}}

\newcommand{\GL}{{\operatorname{GL}}_2}

\newcommand{\M}{{\operatorname{M}}_2}
\newcommand{\JL}{{\operatorname{JL}}}
\newcommand{\ad}{{\operatorname{ad}}}

\newcommand{\vol}{{\operatorname{vol}}}
\newcommand{\OFP}{\mathcal{O}_{F_{\pp}}}
\newcommand{\OKP}{\mathcal{O}_{K_{\pp}}}
\newcommand{\QQ}{\bf{Q}}
\newcommand{\RR}{\bf{R}}
\newcommand{\CC}{\bf{C}}
\newcommand{\ZZ}{\bf{Z}}

\newcommand{\AK}{{\bf{A}}_K}
\newcommand{\AF}{{\bf{A}}_F}
\newcommand{\BB}{\mathcal{B}}
\newcommand{\pp}{\mathfrak{p}}
\newcommand{\PP}{\mathfrak{P}}
\newcommand{\NN}{\mathfrak{N}}

\newcommand{\LL}{\mathcal{L}}
\newcommand{\KK}{\mathcal{K}}

\newcommand{\LLL}{\mathfrak{L}}

\newcommand{\cc}{\mathfrak{c}}

\begin{document}

\title{$p$-adic interpolation of automorphic periods for $\GL$}
\author{Jeanine Van Order}
\thanks{This research was supported in part by the ERC Advanced Grant 247049.}
\subjclass{Primary 11F67; Secondary 11F41, 11F33, 11F70}

\maketitle 

\begin{abstract} 

We give a new and representation theoretic construction of $p$-adic interpolation series for central values of self-dual Rankin-Selberg 
$L$-functions for $\GL$ in dihedral towers of CM fields, using expressions of these central values as automorphic periods. 
The main novelty of this construction, apart from the level of generality in which it works, is that it is completely local. 
We give the construction here for a cuspidal automorphic representation of $\GL$ over a totally real field corresponding to a $\pp$-ordinary 
Hilbert modular forms of parallel weight two and trivial character, although a similar approach can be taken in any setting where the underlying 
$\GL$-representation can be chosen to take values in a discrete valuation ring. 
A certain choice of vectors allows us to establish a precise interpolation formula thanks 
to theorems of Martin-Whitehouse and File-Martin-Pitale. Such interpolation formulae had been conjectured by Bertolini-Darmon in antecedent works. 
Our construction also gives a conceptual framework for the nonvanishing theorems of Cornut-Vatsal in that it describes the underlying theta elements. 
To highlight this latter point, we describe how the construction extends in the parallel weight two setting to give a $p$-adic interpolation series for central 
derivative values when the root number is generically equal to $-1$, in which case the formula of Yuan-Zhang-Zhang can be used to give an 
interpolation formula in terms of heights of CM points on quaternionic Shimura curves. \end{abstract}

\section{Introduction}

Let $F$ be a totally real number field, and let $K$ be a totally imaginary quadratic extension of $F$.
We shall write $\AK$ and $\AF$ to denote the respective adele rings of $K$ and $F$, as well as  
$\eta = \eta_{K/F}: \AF^{\times} /F^{\times} \longrightarrow \lbrace \pm 1 \rbrace \subset {\CC}^{\times} $ 
to denote the quadratic idele class character of $F$ associated to $K/F$. Let $\pi = \otimes_v \pi_v$ be a cuspidal automorphic representation of $\GL(\AF)$. 
We shall assume for simplicity that $\pi$ has trivial central character $\omega = \otimes_v \omega_v = {\bf{1}}$, 
although one could more generally allow $\omega = \eta$. Fix a prime ideal $\pp \subset \mathcal{O}_F$ with underlying rational prime $p$. 
We shall also assume the following simplifying conditions throughout:

\begin{hypothesis}\label{pi} 

The cuspidal representation $\pi = \otimes_v \pi_v$ is a holomorphic discrete series of parallel weight $2$ at each real place of $F$. 
The prime-to-$\pp$-part $\NN'$ of its conductor $\NN = c(\pi) \subset \mathcal{O}_F$ is coprime to the relative discriminant of $K$ over $F$, 
whence $\NN \subset \mathcal{O}_F$ admits a unique factorization in $\mathcal{O}_F$ as 
$\NN = \pp^{\delta} \NN^+ \NN^{-}$, where $\NN^{+}$ denotes the product of primes $v \mid \NN'$ which 
split in $K/F$, and $\NN^{-}$ the product of primes $v \mid \NN'$ which remain inert in $K/F$. 
Assume too that $\delta \in \lbrace 0 ,1 \rbrace$, that $\NN^{-}$ is squarefree, that each prime of $\mathcal{O}_F$
dividing $\NN^+$ splits in $K$, and also that $\pp$ splits in $K$ if $\delta = 1$. \end{hypothesis}

Let $ \Omega: \AK^{\times}/ K^{\times} \longrightarrow {\CC}^{\times}$ be an idele class character of $K$ 
of finite order whose restriction to $\AF^{\times}$ is trivial (i.e. a ring class character), and let 
$\pi(\Omega) = \otimes_v \pi(\Omega)_v$ denote the automorphic representation of $\GL(\AF)$ it induces. 
The representation $\pi \otimes \pi(\Omega)$ is then self-dual, and the corresponding $\GL \times \GL$ Rankin-Selberg $L$-function 
$L(s, \pi \times \Omega) = L(s, \pi \times \pi(\Omega))$ has real-valued coefficients and root number. 
To be more precise, recall that this $L$-function is defined for a complex variable $s \in {\bf{C}}$ with $\Re(s) > 1$ by the absolutely convergent Euler product  
\begin{align}\label{sdrsl} L(s, \pi \times \Omega) &= \prod_{v \in V_F} L(s, \pi_v \times \pi(\Omega)_v) \end{align} 
taken over the set $V_F$ of all places of $F$. It has well-known analytic continuation by the work of Jacquet \cite{Ja} 
and Jacquet-Langlands \cite{JL}, and in this setting satisfies the symmetric functional equation 
\begin{align}\label{sdfe} L(s, \pi \times \Omega) &= \epsilon(s, \pi \times \Omega) L(1-s, \pi \times \Omega).\end{align} 
Additionally, the global root number 
\begin{align}\label{epsilon} \epsilon(1/2, \pi \times \Omega) &= \prod_{v \in V_F} \epsilon(1/2, \pi_v \times \Omega_v) \in {\bf{S}}^1 \end{align}  
appearing in the $\epsilon$-factor $\epsilon(s, \pi \times \Omega) = (c(\pi \times \Omega))^{s- 1/2} \epsilon(1/2, \pi \times \Omega)$ is real-valued, 
hence contained in the set $\lbrace \pm 1 \rbrace = {\bf{R}} \cap {\bf{S}}^1$ of real numbers of complex modulus $1$. 

This root number $ \epsilon(1/2, \pi \times \Omega)$ has the following simple description. 
Given an ideal $\mathfrak{c} \subset \mathcal{O}_{F}$, let $\mathcal{O}_{\mathfrak{c}} = \mathcal{O}_F + \mathfrak{c} \mathcal{O}_K$ 
to denote the $\mathcal{O}_F$-order of conductor $\mathfrak{c}$ in $K$, and 
$\Pic(\mathcal{O}_{\mathfrak{c}})= \AK^{\times}/K^{\times} \widehat{O}_{\mathfrak{c}}^{\times} K_{\infty}^{\times} $ its class group. An idele class
character of $K$ which factors through such a class group $\Pic(\mathcal{O}_{\cc})$ is said to be a {\it{ring class character of conductor $\cc$}}.
In fact, we shall only consider ring class characters of $\pp$-power conductor, and hence those which factor through the profinite group $X$ defined by 
\begin{align*} X = \varprojlim_n X_n = \varprojlim_n \Pic(\mathcal{O}_{\pp^n}). \end{align*} 
The corresponding root numbers in this setting can be described in terms of the decomposition of the conductor $\NN = c(\pi)$ in $K$, 
and are generically independent of the choice of $\Omega$ (see e.g.~\cite[Lemma 1.1]{CV}). 
More precisely, suppose that $\Omega$ is a ring class character factoring through $X$ of conductor $c(\Omega) = \pp^n$ for some integer $n \geq 0$. 
If either (i) $\pp$ does not divide $\NN$ (so that $\delta = 0$), (ii) $\pp$ splits in $K$, or else (iii) the exponent $n$ is sufficiently large, then we have the formula
\begin{align}\label{simplerootnumber} \epsilon(1/2, \pi \times \Omega) &= (-1)^{[F:{\QQ}]} \cdot \eta_{K/F}(\NN'). \end{align} 
Hence, this formula $(\ref{simplerootnumber})$ holds for all ring class characters $\Omega$ of $X$ if either (i), (ii), or (iii) is true;
otherwise, it holds for all but finitely many $\Omega$ factoring through $X$. Thus, we can define $k \in \lbrace 0, 1 \rbrace$ to be the integer determined by the condition
\begin{align}\label{k} \epsilon(1/2, \pi \times \Omega) = (-1)^k \text{   for all but finitely many $\Omega$ factoring through $X$}. \end{align} 
The aim of this article is to give a general and completely local construction of $p$-adic $L$-functions in either case on $k \in \lbrace 0, 1 \rbrace$ 
in this setup, i.e. a general and local construction of a measure on $X$ which interpolates suitably-normalized central values $L(1/2, \pi \times \Omega)$ 
if $k=0$, or else central derivative values $L'(1/2, \pi \times \Omega)$ if $k=1$. 

Let us first consider the case of generic root number $+1$ corresponding to $k=0$. 
In this situation, we know that $\pi$ has a Jacquet-Langlands transfer to the totally definite quaternion algebra $B$ over $F$ whose ramification set $\Ram(B)$ 
consists of prime the divisors of the inert level $\NN^{-}$ (the number of which is $\equiv [F: {\bf{Q}}] \operatorname{mod} 2$), together with the real places of $F$. 
Let us write $\pi' = \otimes_v \pi_v'$ to denote this transfer of $\pi$ to an automorphic representation of $B^{\times}(\AF)$, 
so that the correspondence of Jacquet-Langands gives the relation $\pi = \JL(\pi')$. Let $\Omega$ be a ring class character factoring through $X$.
Let us also fix an embedding of $K$ into $B$, as we can since each place of $\Ram(B)$ is inert in $K$.
The central values $L(1/2, \pi \times \Omega)$ can be related to  the functional $P_{\Omega}^B \in \Hom_{\AK^{\times}}(\pi', \Omega)$ 
defined by the rule sending a vector $\varphi \in \pi'$ to the automorphic period 
\begin{align*} P_{\Omega}^B(\varphi) &= \int_{\AK^{\times}/K^{\times} \AF^{\times}} \varphi(t) \Omega(t) dt. \end{align*} 
Here, the measure $dt$ is defined by putting on $B^{\times}(\AF)$ the product of the local Tamagawa measures times $\zeta_F^S(2)$,
where $S$ is some fixed finite set of places of $F$ containing the real places and the places where $\pi$, $\Omega$, 
and the fixed choice of additive character of $\AF$ are unramified (see \cite[$\S 7$]{FMP}). 
Various results involving calculations with the theta correspondence, 
starting with the landmark theorem of Waldspurger \cite{Wa}, relate these integrals $P_{\Omega}^B(\varphi)$ 
to the central values $L(1/2, \pi \times \Omega)$. Namely, the existence of a vector $\varphi \in \pi'$ for which $P_{\Omega}^B(\varphi) \neq 0$
is shown by Waldspurger to be equivalent to the nonvanishing of the central value $L(1/2, \pi \times \Omega)$. 
A more precise relationship can be established in many cases by making a careful choice of {\it{test vector}} $\varphi \in \pi'$, 
i.e. a choice of decomposable vector $\varphi = \otimes_v \varphi_v$ such that for each place $v$ of $F$ and nonzero local functional 
$l_v \in \Hom_{K_v^{\times}}(\pi_v', \Omega_v)$, the value $l_v(\varphi_v)$ does not vanish 
(cf. \cite{GP}, \cite{MaWh}, \cite{FMP}). Using such a choice of vector $\varphi \in \pi'$, along with the relative 
trace formula (cf. \cite{JC}), the works of \cite{MaWh} and File-Martin-Pitale \cite{FMP} establish the following relationship.
Let us again write $V_F$ to denote the set  of all places of $F$. Let $\Delta_F$ denote the absolute discriminant of $F$, and $\Delta_K$ that of $K$. 
Let $c(\Omega)$ denote the absolute norm of the conductor of $\Omega$ (a power of $p$), 
and $c(\Omega_v)$ the exponent of the conductor of the local character $\Omega_v$ for each $v \in V_F$. 
Let $S_i$ denote the set of places $v \in V_F$ which remain inert in $K$. 
Let $S(\pi)$ denote the set of finite places $v \in V_F$ which divide the level $c(\pi) = \NN$, 
and let $S(\Omega) = \lbrace \pp \rbrace$ the set of finite places $v \in V_F$ where the ring class character $\Omega$ is ramified. 
Writing $c(\pi_v)$ to denote the exponent of the conductor of the local representation $\pi_v$ of $\GL(F_v)$, 
let $S_1(\pi)$ denote the set of places $v \in V_F$ where $c(\pi_v) = 1$, and
$S_2(\pi_v)$ the set of places where $c(\pi_v) \geq 2$. Finally, let $C_v(K, \pi, \Omega)$ the local factor defined in \cite[$\S 4$]{MaWh}. 

\begin{theorem}[File-Martin-Pitale, Martin-Whitehouse]\label{formula+1} 

Assume Hypothesis \ref{pi}, and that $\varphi = \otimes \varphi_v \in \pi'$ is test vector defined in \cite[$\S 7.1$]{FMP} (described below). Then,  

\begin{align*} \frac{\vert P_{\Omega}^B(\varphi) \vert^2}{(\varphi, \varphi )}
= \frac{1}{2} &\left( \frac{\Delta_F}{c(\Omega) \Delta_K} \right)^{\frac{1}{2}}
L_{S(\Omega)}(1, \eta) L_{S(\pi) \cup S(\Omega)}(1, {\bf{1}}_F) L^{S(\pi)}(2, {\bf{1}}_F)  \\ 
&\times \prod_{v \in S(\pi) \cap S(\Omega)^c} e(K_v/F_v) \times \prod_{v \mid \infty} C_v(K, \pi, \Omega) 
\times \frac{L^{S_2(\pi)}(1/2, \pi \times \Omega)}{L^{S_2(\pi)}(1, \pi, \ad)}. \end{align*} 
Here, $e(K_v/F_v)$ denotes the usual ramification index of the local extension $K_v/F_v$, and $(\cdot, \cdot)$ is the standard inner
product on $\pi'$ with respect to the measure on $B^{\times}(\AF)$ given by the product of local Tamagawa measures. \end{theorem}

Keeping with this $k=0$ setup, let us now impose the following conditions at our fixed prime ideal $\pp \subset \mathcal{O}_F$.
Note that the quaternion algebra $B$ over $F$ is split at $\pp$, whence we can and do fix an isomorphism 
$B_{\pp} := B \otimes_F F_{\pp} \approx \M(F_{\pp})$. Note as well that there exists a nonzero decomposable vector 
$\varphi \in \pi'$ whose component $\varphi_{\pp}$ at $\pp$ is fixed by the action of the unit group $R_{\pp}^{\times}$, 
where $R_{\pp}$ is an Eichler order of level $\pp^{\delta}$ in $B_{\pp}$ for $\delta \in \lbrace 0, 1 \rbrace$.
Fix an embedding $\overline{\QQ} \rightarrow \overline{\QQ}_p$. 

\begin{hypothesis}\label{local+1} 

The local representation $\pi_{\pp}'$ is $\pp$-ordinary in the sense that the image under our fixed 
embedding $\overline{\QQ} \rightarrow \overline{\QQ}_p$ of its eigenvalue $a_{\pp} = a_{\pp}(\pi)$ for the Hecke 
operator $T_{\pp}$ defined below is a $\pp$-adic unit. \end{hypothesis} 

The first aim of this paper is to give a general construction of $p$-adic interpolation 
series in this so-called ordinary setting for for the values appearing in Theorem \ref{formula+1}, 
\begin{align*} \mathcal{L}(1/2, \pi \times \Omega) &:= \frac{1}{2} \left( \frac{\Delta_F}{c(\Omega) \Delta_K} \right)^{\frac{1}{2}} L_{S(\Omega)}(1, \eta) 
L_{S(\pi) \cup S(\Omega)}(1, {\bf{1}}_F) L^{S(\pi)}(2, {\bf{1}}_F)  \\ &\times \prod_{v \in S(\pi) \cap S(\Omega)^c}
e(K_v/F_v) \times \prod_{v \mid \infty} C_v(K, \pi, \Omega) \times \frac{L^{S_2(\pi)}(1/2, \pi \times \Omega)}{L^{S_2(\pi)}(1, \pi, \ad)}. \end{align*}
Let us remark that many of the conditions imposed in Hypothesis \ref{pi} can be lifted at the expense of clarity, and that Hypothesis \ref{local+1}
above is the main requisite condition to impose for this construction so long as the decomposable vector $\varphi \in \pi'$ can be chosen in such 
a way as to take values in a discrete valuation ring $\mathcal{O}$. 
Assuming as we do that $\pi$ is a holomorphic discrete series of parallel weight two, 
and hence that its Hecke eigenvalues define algebraic numbers by a theorem of Shimura \cite{Sh2}, 
we let $\mathcal{O}$ denote the ring of integers of a fixed extension of ${\QQ}_p$ containing the Hecke field ${\QQ}(\pi)$. 
Consider the $\mathcal{O}$-Iwasawa algebra of $X$, 
\begin{align*} \mathcal{O}[[X]] &= \varprojlim_n \mathcal{O}[X_n]. \end{align*} The elements of this group algebra $\mathcal{O}[[X]]$ 
can be identified with their corresponding $\mathcal{O}$-valued measures on $X$ (see e.g. \cite[$\S 7$]{MSD}). 
We shall make such an identification implicitly throughout the rest of this work. 

\begin{theorem}[Proposition \ref{interpolation}, Corollary \ref{ninterpolation}] 

Suppose that $\pi$ satisfies Hypotheses \ref{pi} and \ref{local+1}, and that we are in the setting of $k=0$ on the generic root number described in $(\ref{k})$. 
There exists a nontrivial element $\mathfrak{L}_{\pp}(\pi', K) \in \mathcal{O}[[X]]$ whose specialization $\Omega(\mathfrak{L}_{\pp}(\pi', K)) = \int_{X_n} 
\Omega(\sigma) d\mathfrak{L}_{\pp}(\pi', K)(\sigma)$ to $\Omega$ any character of $X$ of conductor $\pp^n$ with exponent $n \geq 1$ 
satisfies the following interpolation formula:
\begin{align}\label{interpolation} \Omega(\mathfrak{L}_{\pp}(\pi', K)) 
&= \alpha_{\pp}^{2(\delta - 1 -n)} \cdot \left( \frac{h(\mathcal{O}_F)}{m(\mathcal{O}_{\pp^n})}\right)^2 \cdot \mathcal{L}(1/2, \pi \times \Omega)  \in \overline{\QQ}_p. \end{align}
Here, writing $q_{\pp}$ to denote the cardinality of the residue field of $\pp$, $\alpha_{\pp} = \alpha_{\pp}(\pi)$ denotes the unit root of the Hecke 
polynomial $t^2 - a_{\pp} t + q_{\pp}$ if $\delta = 0$ and otherwise the eigenvalue of the Hecke operators $T_{\pp}^l$ and $T_{\pp}^{u}$ defined below if $\delta =1$.  
As well, $h(\mathcal{O}_F)$ denotes the class number of $F$, and $m(\mathcal{O}_{\pp^n})$ the volume of $\widehat{\mathcal{O}}_{\pp^n}^{\times}$ in 
$K^{\times} \backslash \AK^{\times} / \AF^{\times}$ with respect to our fixed choice of Haar measure. \end{theorem}

This result extends those given by Bertolini-Darmon in \cite{BD96} and \cite{BD} for the totally real field $F= {\QQ}$, as well as previous work of 
the author for the general totally real fields setting (\cite{VO1}, \cite{VO2})\footnote{The existence of such an interpolation series is by no means 
new however, and has been given by completely different methods e.g. in Hida \cite{Hi}, \cite{Hi91} (cf. \cite{PR88}) and Haran \cite{Har} (cf. \cite{Lo}).}. 
Our construction differs from these antecedent works though, and proceeds by choosing a sequence of local test vectors. This allows us to give a more 
precise interpolation formula. The novelty of this construction, apart from the scope in which it works (and the representation theoretic setup), is that it is completely local to $\pp$. 
This allows us for instance to use test vectors according to \cite{MaWh} and \cite{FMP} to derive an even more precise interpolation formula than what 
was conjectured (in a special, classical case) by Bertolini-Darmon in \cite{BD96}. 

The second aim of this work is to extend the construction of the elements $\mathfrak{L}_{\pp}(\pi', K) \in \mathcal{O}[[X]]$ for Theorem \ref{ninterpolation} 
to the dichotomous setting of $k=1$ on the generic root number, i.e.~so that $\epsilon(1/2, \pi \times \Omega) = -1$ for all but finitely many ring class 
characters $\Omega$ factoring through the profinite group $X = \varprojlim_n X_n$. Since the central values $L(1/2, \pi \times \Omega)$ 
in this setting vanish as by the symmetric functional equation $(\ref{sdfe})$, we propose to construct a $p$-adic interpolation series 
for the central first derivative values $L'(1/2, \pi \times \Omega)$, where the issue of defining suitable periods is a subtler problem. 
Our construction here depends in a crucial way on the those given implicitly in the work of Yuan-Zhang-Zhang \cite{YZZ}. 
To state our main result for this second part, we shall first need to describe the central derivative formula of \cite[Theorem 1.2]{YZZ} 
(Theorem \ref{YZZ} below). Let us for now give a preliminary sketch of their result, saving details for later.

Let $A$ be a principally polarized abelian variety of $\GL$-type parametrized by a quaternionic Shimura curve $M = \lbrace M_H \rbrace_H$ defined over $F$. 
Thus, $A$ is defined over $F$, the endomorphism algebra $\End^0(A) = \End(A) \otimes {\QQ}$ defines a number field $L$ of degree equal to the
dimension of $A$, and there exists a nonconstant morphism $M_H \longrightarrow A$ for $H$ some compact open subgroup of the unit group defined by
$\widehat{\BB}^{\times}$. Here, $\BB$ is the ambient quaternion algebra defined over $F$ associated to the Shimura curve $M$, 
which is ramified at each of the real places of $F$ except for one fixed place $\tau$, as well as each of the places dividing the inert level $\NN^{-}$.  
Hence, we can and do fix an embedding of $K$ into $\BB$. Writing $J_H$ to denote the Jacobian of the curve $M_H$, we consider the space defined by 
\begin{align*} \pi_A &= \Hom^0(J, A) = \varinjlim_H \Hom_F(J_H, A) \otimes_{\ZZ} {\QQ}, \end{align*} 
as constructed in \cite[$\S 3.2$, Theorem 3.8]{YZZ} (cf. \cite[$\S 1.2.3$]{YZZ}). 
To describe this briefly, $\pi_A$ can be viewed as an automorphic representation of $\widehat{\BB}^{\times}$, (see \cite[$\S 1.2.1$]{YZZ}).
This automorphic representation $\pi_A$ admits a natural 
decomposition $\pi_A = \otimes_v \pi_{A, v}$ into irreducible admissible representations $\pi_{A, v}$ of the local unit groups $\BB_v^{\times}$ over $F_v$.  
Moreover (as explained in \cite{YZZ}), this representation can be viewed as a geometric realization of the nonarchimedean component of the 
Jacquet-Langands transfer of $\pi$. The conditions of Hypothesis \ref{pi} then imply that $\pi_A$ is self-dual, with trivial central character. 
The main theorem of Yuan-Zhang-Zhang \cite{YZZ} relates the central derivative value $L'(1/2, \pi_A \times \Omega)$ to a certain canonical 
generator $\alpha$ of the space $\mathcal{P}(\pi_A, \Omega) \otimes \mathcal{P}(\pi_A, \Omega^{-1})$, where
$\mathcal{P}(\pi_A, \Omega) := \Hom_{\AK^{\times}}(\pi_A \otimes \Omega, L)$. We refer to the discussion in \cite[$\S 1$]{YZZ} for more details. 
In brief, this generator $\alpha$ decomposes into a product 
$\otimes_v \alpha_v$ of local generators $\alpha_v$ in the analogously-defined local spaces 
$\mathcal{P}(\pi_{A, v}, \Omega_v) \otimes \mathcal{P}(\pi_{A, v}, \Omega_v^{-1})$, 
and these local generators can be described more precisely as follows.\footnote{Note that we have already used the symbol 
$\alpha_{\pp}$ to denote the unit root of the Hecke polynomial $t^2 - a_{\pp}t + q_{\pp}$ 
in the discussion above; we trust that the distinction between these notations should be clear.} 
Let us also for simplicity write $\pi_v$ to denote the local representation $\pi_{A, v}$ of $\BB_v^{\times}$. 
Let us for each place $v$ of $F$ fix a Haar measure $dt_v$ on $K_v^{\times}/F_v^{\times}$ such that 
(i) the product measure defined over all places $v$ of $F$ is the Tamagawa measure on $\AK^{\times}/K^{\times}$ and 
(ii) the maximal compact subgroup $\OKP^{\times}/\OFP^{\times}$ is assigned as a volume some rational number. Then, for given 
vectors $\varphi_{1,v} \in \pi_v$ and $\varphi_{2,v} \in \pi_v^{\vee}$, the local component $\alpha_v$ is defined formally by the ratio of values
\begin{align*} \alpha(\varphi_{1,v}, \varphi_{2, v}) &:= \frac{ L(1, \eta_v) L(1, \pi_v, \ad)}{\zeta_{F_v}(2)L(1/2, \pi_v \times \Omega_v)}
\int_{K_v^{\times}/F_v^{\times}} (\pi_v(t) \varphi_{1, v}, \varphi_{2, v})_v \Omega_v(t) dt. \end{align*} 
Fixing an embedding $\iota: L \rightarrow {\CC}$, this local integral can be seen to take values in $L$ independently of the choice of embedding $\iota$, 
as explained in \cite[$\S 1.3$]{YZZ}. The local invariant pairings $(\cdot, \cdot)_v: \pi_v \times \pi_v^{\vee} \longrightarrow R$ also satisfy the compatibility 
relation $(\cdot, \cdot) = \otimes_v (\cdot, \cdot)_v$, where $(\cdot, \cdot): \pi_A \times \pi_A^{\vee} \longrightarrow R$ denotes the perfect 
$\BB^{\times}(\AF)$-invariant pairing introduced below. As explained in \cite{YZZ} (and summarized below), 
there is for given a vector $\varphi \in \pi_A$ an analogous notion of an associated automorphic period 
$P_{\Omega}^{\BB}(\varphi) \in A(K^{\ab}) \otimes_R L$, where $K^{\ab}$ denotes the maximal abelian extension of $K$. 
The following formula, which is the main result of the \cite{YZZ}, relates this period and the generator $\alpha = \otimes \alpha_v$ 
to the central derivative values $L'(1/2, \pi_A \times \Omega)$.

\begin{theorem}[Yuan-Zhang-Zhang]\label{YZ^2} 

Given decomposable vectors $\varphi_1 \in \pi_A$ and $\varphi_2 \in \pi_A^{\vee}$
in the setup described above, we have the following identity in $\mathcal{O}_L \otimes_{\QQ} \CC$: 
\begin{align*} \langle P_{\Omega}^{\BB}(\varphi_1), P^{\BB}_{\Omega^{-1}}(\varphi_2) \rangle_L 
&= \frac{\zeta(2)L'(1/2, \pi_A \times \Omega)}{4 L(1, \eta)^2 L(1, \pi_A, \ad)}
\cdot \alpha(\varphi_1, \varphi_2). \end{align*}  \end{theorem}

\begin{proof} The result is a consequence of \cite[Theorem 1.2]{YZZ}, see Theorem \ref{YZZ} below.\end{proof}

Using this result, we construct the following measure. This construction requires the choosing a certain class
$x_n \in \BB_{\pp}^{\times}/R_{\pp}^{\times}$ for each integer $n \geq 1$ to obtain suitable distribution relations. 
We commit a minor abuse of notation in also writing $x_n$ to denote the associated adele class in 
$\widehat{\BB}^{\times}/H$, for $H = H^{\pp} R_{\pp}^{\times}$ a suitably-defined compact open subgroup of $\widehat{\BB}^{\times}$. 
Let us also now write $\mathcal{O}$ to denote the tensor product 
$\mathcal{O} = A(K^{\ab}) \otimes_{\ZZ} {\ZZ}_p$.

\begin{theorem}[Proposition \ref{interpolation-1}] 

Suppose that $\pi$ satisfies Hypotheses \ref{pi} and \ref{local+1}, and that we are in the setting of $k=1$ on the generic root number described in $(\ref{k})$. 
Suppose as well that $A/F$ is the abelian variety of $\GL$-type for which the corresponding representation $\pi_A$ is a Jacquet-Langlands transfer 
of the finite part $\pi^{(\infty)} = \otimes_{v \nmid} \pi_v$ of $\pi$ to the quaternion algebra $\BB$, i.e.~$\JL(\pi_A) = \pi^{(\infty)}$.  
There exists a nontrivial element $D_{\pp}^{(\delta)}(\varphi, K)  = (D_{\pp}^{(\delta)}(\varphi, K)(x_n))_{n \geq 1}$ in $\mathcal{O}[[X]]$, depending 
on the choice of sequence of local classes $(x_n)_{n \geq 1}$, whose specialization after composition with the N\'eron-Tate height pairing 
$\langle ~ , ~ \rangle_L$, i.e. $$\langle ~,~ \rangle_L \circ \Omega( D_{\pp}^{(\delta)}(\varphi, K)) 
= \langle ~,~ \rangle_L \circ \int_X \Omega(\sigma) d  D_{\pp}^{(\delta)}(\varphi, K) (\sigma), $$ to any ring class character $\Omega$ of $X$ of conductor 
$\pp^n$ with $n \geq 1$ satisfies the following interpolation formula: 
\begin{align*} \langle ~,~\rangle_L &\circ \Omega(D_{\pp}^{(\delta)}(\varphi, K)) \\
&= \alpha_{\pp}(\pi_A)^{-2(1-\delta - n)} \cdot \left( \frac{h(\mathcal{O}_F)}{m(\mathcal{O}_{\pp^n})}\right)^2 
\cdot \frac{\zeta(2) L'(1/2, \pi_A \times \Omega)}{4 L(1, \eta)^2 L(1, \pi_A, \ad)} \cdot  
\alpha (x_n \cdot \varphi, x_n \cdot \varphi). \end{align*} 
Here, $\alpha_{\pp}(\pi_A)$ denotes the unit root of the Hecke polynomial $t^2 - a_{\pp} t + q_{\pp}.$ \end{theorem}

\begin{remark}[Outline of article.] We first give some more details of the central value formulae described above in $\S 2$, 
and then the distribution relations of Cornut-Vatsal \cite[$\S 6$]{CV} in $\S3$, before giving the two main constructions of $p$-adic interpolation series in $\S 4$. \end{remark}

\subsection{Notations} Given a prime $v$ of $\mathcal{O}_F$, we let $F_v$ to denote the completion of $F$ at $v$, and $\mathcal{O}_{F_v}$ the ring of integers of $F_v$ 
if $v$ is finite. If $E$ is either a quadratic extension of $F$ or a quaternion algebra over $F$, then we write $E_v = E \otimes_F F_v$. If $R$ is a module over the 
ring of integers $\mathcal{O}_F$ of $F$, then we write $R_v = R \otimes_{\mathcal{O}_F} \mathcal{O}_{F_v}$. We write ${\bf{A}}$ to denote the ring of adeles of ${\bf{Q}}$, with 
${\bf{A}}_f \times {\bf{R}}$ to denote the decomposition into the finite adeles ${\bf{A}}_f$. We then write ${\bf{A}}_F = {\bf{A}} \otimes_{\bf{Q}} F$ and 
${\bf{A}}_K = {\bf{A}} \otimes_{\bf{Q}} K$. We also use the standard hat notation, e.g. $\widehat{F} = {\bf{A}}_f \otimes_{\bf{Q}} F$. 
Thus, $\widehat{F} = \widehat{O}_F \otimes {\bf{Q}}$, where $\widehat{\mathcal{O}}_F = \mathcal{O}_F \otimes \widehat{\bf{Z}}$ is the profinite completion of $\mathcal{O}_F$. 
Given $M$ a finitely generated ${\bf{Z}}$-module, we also write $\widehat{M} = M \otimes \widehat{\bf{Z}}$ to denote 
the profinite completion of $M$. Notations for $L$-values are adopted from the sources (namely \cite{CV}, \cite{MaWh}, \cite{FMP}, and \cite{YZZ}).

\section{Central value formulae and algebraicity}

We now include some details about the special value formulae mentioned above. 

\subsection{Central values for the case of $k=0$}

Suppose first that we are in the setting of $k=0$ described above. Hence, $\epsilon(1/2, \pi \times \Omega) = +1$ for all but 
finitely many ring class characters factoring through $X$, in which case the corresponding values $L(1/2, \pi \times \Omega)$ 
are not forced to vanish by the functional equation.
Recall that in this setting there exists a quaternion algebra $B$ over $F$ such that (i) $K$ embeds into $B$, 
(ii) $\pi = \otimes_v \pi_v$ transfers to an automorphic representation $\pi' = \otimes \pi_v'$ on $B^{\times}({\bf{A}}_F)$, 
and (iii) $\Hom_{{\bf{A}}_K^{\times}}(\pi', \Omega)$ is not identically zero. 
In this setting, we use the formulae of \cite{MaWh} and \cite{FMP} described 
in Theorem \ref{formula+1} to relate the values $L(1/2, \pi \times \Omega)$ to the 
automorphic periods $P^B_{\Omega}(\varphi)$. Let us thus define the test vectors 
and local factors appearing in the formula of Theorem \ref{formula+1} above to give a complete description, 
and then describe briefly the algebraicity properties satisfied by these values. 

\subsubsection{Choices of vectors} 

Given a prime $v \in V_F$ which is inert in $K$, and for which $c(\pi_v), c(\Omega_v) >0$, let us assume that $c(\Omega_v) \geq c(\pi_v)$. 
We shall later choose a sequence of local vectors $\varphi_{\pp, n}$ in $\pi_{\pp}'$ for our construction. Since we wish to use the result of Theorem 
\ref{formula+1} to derive a precise interpolation formula here, we shall have to ensure that our local vectors are test vectors in the sense 
of \cite{FMP} (cf. \cite{GP}). According to \cite[Theorem 1.7, $\S 1.2.3$]{FMP}, these test vectors can be characterized as follows.

\begin{theorem}[File-Martin-Pitale]\label{1.7} 

Given a prime ideal $v \subset \mathcal{O}_F$, let $\PP_v \subset \mathcal{O}_{F_v}$ be the maximal ideal, and fix a uniformizer $\varpi_v$. 
Let $\pi_v$ be an irreducible admissible representation of $\GL(F_v)$ of conductor $\PP_v^{c(\pi_v)}$ and trivial central character. 
Let $\Omega_v$ be a character of $K_v^{\times}$ for which the restriction $\Omega_v \vert_{F_v^{\times}}$ to $F_v^{\times}$ is trivial. 
Assume that $c(\Omega_v) \geq c(\pi_v) > 0$. Viewing $K_v^{\times}$ as a torus of $\GL(F_v)$ via the embedding defined in 
\cite[$\S 2.3$]{FMP}, we have that $\dim_{\CC}( \Hom_{K_v^{\times}}(\pi_v, \Omega_v)) = 1$. Moreover, writing
\begin{align*} \KK_1(\PP_v^n) &=\left\lbrace \left(\begin{array} {cc} a&  b \\ c & d \end{array}\right) \in 
\GL(\mathcal{O}_{F_v}) : c \in \PP_v^n, d \in 1 + \PP_v^n \right\rbrace, \end{align*} and 
\begin{align*} h &= \left(\begin{array} {cc} \varpi_v^{c(\Omega_v) - c(\pi_v)}&  0 \\ 
0 & 1 \end{array}\right) \left(\begin{array} {cc} 0 & 1 \\ -1 & 0 \end{array}\right) = \left(\begin{array} {cc} 0 & \varpi_v^{c(\Omega_v) - c(\pi_v)} \\ 
-1 & 0 \end{array}\right) , \end{align*} the following is true: Given $l_v \in \Hom_{K_v^{\times}}(\pi_v, \Omega_v)$ a nonzero functional, the subgroup 
$ h \KK_1(\PP_v^{c(\pi_v)}) h^{-1} \subset  \GL(\mathcal{O}_{F_v})$ fixes a $1$-dimensional subspace of $\pi_v$ consisting of test vectors for $l_v$. \end{theorem}

Taken with the construction of Gross-Prasad \cite{GP} (for $c(\Omega_v) \geq c(\pi_{v}) = 0$), this has the following interpretation (see \cite[$\S 1.2.3$]{FMP}). 
Let $R_v$ be an order of $\M(F_v)$. Let $d(R_v)$ be the exponent of its reduced 
discriminant. Let $c(R_v)$ denote the exponent of its conductor, i.e. the smallest integer $c \geq 0$ for which $\mathcal{O}_{F_v} + \varpi_v^c \mathcal{O}_{K_v} 
\subseteq R_v.$ The local order $R_v$ can fix a test vector only if $c(R_v) \geq c(\Omega_v)$. 
Intuitively, it seems reasonable to expect that if such a local order 
$R_v$ satisfies the conditions $c(R_v) = c(\Omega_v)$ and $d(R_v) = c(\pi_v)$, then it might select a local test vector in $\pi_v$, i.e. so that the subgroup $R_v^{\times}$ 
of $\GL(F_v)$ fixes a $1$-dimensional space of test vectors for any nonzero linear functional $l_v \in \Hom_{K_v^{\times}}(\pi_v, \Omega_v)$. 
If it is the case
that either $c(\pi_v)$ or $c(\Omega_v)$ equals $0$, then it is shown in \cite{GP} that such an order $R_v$ exists, and that is unique up to conjugation by $K_v^{\times}$. 
Moreover, this order is maximal if $c(\pi_v) = 0$, and Eichler if $c(\pi_v) >0$. If it is the case that $c(\Omega_v) \geq c(\pi_v) >0$, then the invariants $d(R_v)$ and $c(R_v)$
no longer specify such an order $R_v$ uniquely. However, the result of Theorem \ref{1.7} can be used to give the following interpretation of the problem, as explained in 
File-Martin-Pitale \cite[$\S 1.2.3$]{FMP}. Namely, in the setting where $c(\Omega_v) \geq c(\pi_v) >0$, there exists an Eichler order $R_v$ of $\M(F_v)$ with 
$c(R_v) = c(\Omega_v)$ and $d(R_v) = c(\pi_v)$ for which the subgroup $R_v^{\times}$ of $\GL(F_v)$ fixes a line or test vectors in $\pi_v$. Moreover, this order 
$R_v$ can be expressed uniquely at the intersection or two maximal orders $R_{1, v}$ and $R_{2, v}$ for which $c(R_{1, v}) = c(\Omega_v)$ and 
$c(R_{2, v}) = c(\Omega_v) - c(\pi_v)$. We shall use these results of \cite{GP} and \cite{FMP} below, specifically in the setting where $v = \pp$, $n = c(\Omega_{\pp}) \geq 0$, 
and the exponent $\delta = c(\pi_{\pp})$ equals $0$ or $1$ respectively. 

\subsubsection{Algebraicity} 

We can deduce the following algebraicity result independently of the theorem Shimura \cite{Sh2}.
Let $\pi^{\sigma}$ denote the representation of $\GL(\AF)$ defined by the rule that sends the eigenvalues of $\pi$ to their images 
under $\sigma \in \Aut({\bf{C}})$. Let $\Omega^{\sigma}$ denote the character defined on nonzero ideals $\mathfrak{a} \subset \mathcal{O}_K$ 
by the rule $\mathfrak{a} \mapsto \Omega(\mathfrak{a})^{\sigma}$.

\begin{corollary}\label{alg+1} 

Let $L = {\QQ}(\pi, \Omega)$ be the finite extension of ${\bf{Q}}$ obtained by adjoining the eigenvalues of $\pi$ and the values of $\Omega$. 
If $B$ is totally definite, then the values 
\begin{align*} \LL(1/2, \pi \times \Omega) & =  \frac{1}{2} \left( \frac{\Delta_F}{c(\Omega) \Delta_K} \right)^{\frac{1}{2}} L_{S(\Omega)}(1, \eta) 
L_{S(\pi) \cup S(\Omega)}(1, {\bf{1}}_F) L^{S(\pi)}(2, {\bf{1}}_F)  \\ &\times \prod_{v \in S(\pi) \cap S(\Omega)^c} e(K_v/F_v) 
\times \prod_{v \mid \infty } C_v(K, \pi, \Omega) \times \frac{L^{S_2(\pi)}(1/2, \pi \times \Omega)}{L^{S_2(\pi)}(1, \pi, \ad)} \end{align*} 
are algebraic, and moreover lie in the number field $L \subset \overline{\QQ}$. Moreover, there is a natural action 
of $\sigma \in \Aut({\CC})$ on these values given by the rule $ \LL(1/2, \pi \times \Omega)^{\sigma} = \LL(1/2, \pi^{\gamma} \times \Omega^{\sigma}), $ where 
$\gamma$ denotes the restriction of $\sigma$ to the Hecke field ${\bf{Q}}(\pi)$ of $\pi$. \end{corollary}

\begin{proof} 

By Theorem \ref{formula+1}, we have that \begin{align}\label{ev0} \LL(1/2, \pi \times \Omega) & = \vert P_{\Omega}^B(\varphi) \vert^2/(\varphi, \varphi). \end{align} 
Now, we can view $\varphi$ as an automorphic function on the set $C(B; H) = B^{\times} \backslash \widehat{B}^{\times} /H$, for $H$ some compact open subgroup 
of $\widehat{B}^{\times}$. This function is determined uniquely up to multiplication by a nonzero complex scalar by the Jacquet-Langlands correspondence. 
Moreover, since $B$ is totally definite, the space $C(B; H)$ is finite. Using this fact, it is easy to see that we can fix a basis for the space of automorphic forms 
on $C(B; H)$ taking values in ${\ZZ}(\pi)$. The result is then simple to deduce from $(\ref{ev0})$, using the fact that $P_{\Omega}^B(\varphi)$ is a finite integral. \end{proof} 

\subsection{Central derivative values for the case of $k=1$}

Let us now suppose that $k=1$. Hence, $\epsilon(1/2, \pi \times \Omega) = -1$ for all but finitely many ring class characters $\Omega$ factoring through $X$,
and the corresponding central values $L(1/2, \pi \times \Omega)$ are forced to vanish by the functional equation. In this setting, we have the following formula 
of Yuan-Zhang-Zhang \cite{YZZ} (generalizing Gross-Zagier \cite{GZ}) for the derivative values $L'(1/2, \pi \times \Omega)$. 
We now describe this in more detail, following \cite[$\S 1.2,1.3, 3.2$]{YZZ}.

Fix a real place $\tau$ of $F$. Let $\BB$ denote the quaternion algebra defined over $F$ which is split at $\tau$, 
but ramified at each other real place, as well as the finite places dividing the inert level $\NN^{-}$. 
Again, we can and do fix an embedding of $K$ into $\BB$.

\subsubsection{Shimura curves}

Let $\widehat{\BB}^{\times}$ denote the finite adelic points of $\BB^{\times}(\AF)$. 
Given $H \subset \widehat{\BB}^{\times}$ any compact open subgroup, let $M_H$ 
denote the associated (compactified) Shimura curve over $F$, 
whose complex points determine a Riemann surface \begin{align*} 
M_{H, \tau}(\CC) &= \BB^{\times} \backslash \mathfrak{H}^{\pm} \times
\widehat{\BB}^{\times} / H \times \lbrace \operatorname{cusps} \rbrace. \end{align*} 
Note that the set of cusps  $\lbrace \operatorname{cusps} \rbrace $ is nonempty only if $F = \QQ$. 

Given compact open subgroups $H_1, H_2 \subset \widehat{\BB}^{\times}$ for which the inclusion
$H_1 \subset H_2$ holds, there is a natural and surjective morphism \begin{align*}
\pi_{H_1, H_2}: M_{H_1} &\longrightarrow M_{H_2}. \end{align*} We shall write 
$M = \lbrace M_H \rbrace_H$ to denote the associated projective system. Note 
that each $M_H$ can be identified with the quotient of $M$ by the action of $H$.

\subsubsection{Hodge classes}

A {\it{Hodge class}} on a Shimura curve $M_H$ is a line bundle $L_H$
on $\Pic(M_H)_{\QQ}$ whose global sections are holomorphic forms for 
parallel weight $2$. We refer to \cite[$\S 1.2$]{YZZ} for a more explicit 
description of these classes. Given a connected component $\beta \in 
\pi_0(M_{H, \overline{F}})$, we write $L_{H, \beta} = L_H \vert_{M_{H, \beta}}$ 
to denote the restriction of $L_H$ to the connected component $M_{H, \beta}$ 
of $M_{H}$ corresponding to $\beta$. We then view this $L_{H, \beta}$ as a 
divisor on $M_H$ via the pushforward under $M_{H, \beta} \longrightarrow M_H$.

\begin{definition} Given a Hodge class $L_H$ and a connected component
$\beta \in \pi_0(M_{H, \beta})$, the {\it{normalized Hodge class on $M_{H, \beta}$}} 
is the weighted class $\xi_{H, \beta} = L_{H, \beta} / \deg(L_{H, \beta})$. The 
{\it{normalized Hodge class on $M_{H}$}} is the sum $\xi_H = \sum_{\beta} \xi_{H, \beta}$. \end{definition}

\subsubsection{Abelian varieties parametrized by Shimura curves}

Let $M = \lbrace M_H \rbrace_H$ be a Shimura curve defined over $F$, as above.
A simple abelian variety $A$ defined over $F$ is said to be {\it{parametrized by $M$}} 
if for some compact open subgroup $H \subset \widehat{\BB}^{\times}$, there exists a 
non constant morphism $M_H \longrightarrow A$ defined over $F$. It is known 
by Eichler-Shimura theory that if such an abelian variety $A$ parametrized by 
$M$ is of {\it{strict $\GL$-type}} in the sense that (i) $R = \End_F(A) \otimes_{\ZZ} {\QQ}$
is a field and (ii) $\operatorname{Lie}(A)$ is a free module of rank $1$ over 
$R \otimes_{\QQ} F$ by the induced action. Given such an abelian variety $A$, let us consider the space
\begin{align*} \pi_A &= \Hom_{\xi}^0(M, A) = \varinjlim_H \Hom_{{\xi}_U}^0(M_H, A).
\end{align*} Here, each $\Hom_{{\xi}_H}^0(M_H, A)$ denotes the morphisms in the space
$$\Hom^0(M_H, A) = \Hom_F(M_H A) \otimes_{\ZZ} \QQ$$ having basepoint equal to the Hodge 
class $\xi_H$. This $\pi_A$ has the natural structure of a $\widehat{\BB}^{\times}$-module. Moreover, as shown 
in \cite[Theorem 3.8]{YZZ}, this space $\pi_A$ in fact determines an ($A(\overline{F})_{\QQ} = A(\overline{F}) \otimes_{\ZZ} {\QQ}$-valued) 
automorphic representation of $\widehat{\BB}^{\times}$ over $\QQ$. The description of \cite[$\S 3.2.2$]{YZZ} also shows that there is a natural 
identification $R = \End_{\widehat{\BB}^{\times}}(\pi_A)$, as well as a decomposition $\pi_A =\otimes_v \pi_{A, v}$ into absolutely irreducible 
representations $\pi_{A, v}$ of $\BB_v^{\times}$ over $R$. Finally, since any morphism 
$M_H \longrightarrow A$ factors through the Jacobian $J_H$ of $M_H$, 
we can and do redefine this automorphic representation in the simpler form 
\begin{align*} \pi_A &= \Hom^0(J, A) = \varinjlim_H \Hom^0(J_H, A).
\end{align*} Here, in the same style as above, we have put 
$\Hom^0(J_H, A) = \Hom_F(J_H, A) \otimes_{\ZZ} {\QQ}$. 

\subsubsection{Dual abelian varieties}

Let us write $A^{\vee}$ to denote the dual of an abelian variety $A$. If $A$ is parametrized
by a Shimura curve $M =\lbrace M_U \rbrace_U$ in the sense defined above, then so too 
is $A^{\vee}$. Moreover, writing $R^{\vee} = \End_F(A^{\vee}) \otimes_{\ZZ} {\QQ}$, there is 
a canonical isomorphism $R \approx R^{\vee}$ obtained by sending a homomorphism 
$r: A \longrightarrow A$ to its dual $r^{\vee}: A^{\vee} \longrightarrow A^{\vee}$.

\subsubsection{Pairings}

Let us first consider the perfect, $\BB^{\times}(\AF)$-invariant pairing 
\begin{align*} (\cdot , \cdot ): \pi_A \times \pi_{A^{\vee}} &\longrightarrow R \end{align*} 
defined on elements $\varphi_{1, H} \in \Hom(J_H, A)$ and $\varphi_{2, H} \in \Hom(J_H, A)$ by 
\begin{align*} (\varphi_1, \varphi_2) &= \left( \varphi_{1, H} \circ \varphi_{2, H}^{\vee} \right)/\vol(M_H) .\end{align*}  
Here, $J_H$ denotes the Jacobian of the Shimura curve $M_H$, and $\varphi_{2, H}^{\vee}: A \longrightarrow J_H$ 
the dual of $\varphi_{2, H}$ composed with the canonical isomorphism $J_H^{\vee} \approx J_H$. 
This description implies that $\pi_{A^{\vee}}$ is in fact the dual of $\pi_A$ as a representation of $\BB^{\times}(\AF)$ over $R$.

Let us also consider the N\'eron-Tate height pairing (as defined e.g. in \cite[$\S 7$]{YZZ}), 
which recall is a ${\QQ}$-bilinear non-degenerate pairing \begin{align*} \langle \cdot, 
\cdot \rangle: A(\overline{F})_{\QQ} \times A^{\vee}(\overline{F})_{\QQ} \longrightarrow {\RR}. 
\end{align*} Here, we written $A(\overline{F})_{\QQ} = A(\overline{F}) \otimes_{\ZZ} {\QQ}$ 
and $A^{\vee}(\overline{F})_{\QQ} = A^{\vee}(\overline{F}) \otimes_{\ZZ} {\QQ}$ 
to denote the tensor products appearing throughout \cite{YZZ}. Of course, the 
N\'eron-Tate pairing is defined classically on the groups $A(\overline{F})$ and 
$A^{\vee}(\overline{F})$, and then extended in a natural way to these tensor products. 
This extended N\'eron-Tate pairing in fact descends to a ${\QQ}$-linear map 
\begin{align*} \langle \cdot, \cdot \rangle: A(\overline{F})_{\QQ} \otimes_R A^{\vee}(\overline{F})_{\QQ} 
 \longrightarrow {\RR}. \end{align*} Some more explanation of this fact is given in 
 \cite[$\S 1.2.3$]{YZZ} (cf. \cite[Proposition 7.3]{YZZ}). 
 
Now, given $a \in R$, $x \in A(\overline{F})_{\QQ}$, and $y \in A^{\vee}(\overline{F})_{\QQ}$, the rule 
$a \mapsto \langle ax, y \rangle$ defines an element of the space $\Hom(R, {\RR})$. As explained in
\cite[$\S 1.2.4$]{YZZ}, the trace map can be used to construct an isomorphism 
$\Hom(R, {\RR}) \approx R \otimes_{\QQ} {\RR}$, whence we write $\langle x, y \rangle_R$
to denote the corresponding element of the space $R \otimes_{\QQ} {\RR}$. This gives us 
the construction of an $R$-bilinear pairing \begin{align*} \langle \cdot, \cdot \rangle_R: 
A(\overline{F})_{\QQ} \otimes_R A^{\vee}(\overline{F})_{\QQ} &\longrightarrow
R \otimes_{\QQ} {\RR} \end{align*} for which \begin{align*} \langle \cdot, \cdot \rangle &= 
\Tr_{R \otimes {\RR}/ {\RR}} ~\langle \cdot, \cdot \rangle_R. \end{align*} We shall refer to this
 $R$-linear pairing as the {\it{$R$-linear N\'eron-Tate pairing}}. 
 
\subsubsection{CM points}

Recall that we fix an embedding $K \rightarrow \BB$. The adele group $\AK^{\times}$ acts on $\lbrace M_H \rbrace_H$ 
by right multiplication via the embedding $\AK^{\times} \rightarrow \BB^{\times}(\AF)$. 
Let $M^{K^{\times}}$ denote the subscheme of $M$ fixed by $K^{\times}$ under this action. 
By the theory of complex multiplication, each point of $M^{K^{\times}}(\overline{F})$ is defined over the maximal abelian extension $K^{\ab}$ of $K$. 
Let us now fix a point $Q \in M^{K^{\times}}(K^{\ab})$, and hence a point 
$Q_H \in M_H(K^{\ab})$ for each compact open subgroup $H \subset \widehat{\BB}^{\times}$. We shall normalize 
\begin{align*} M_{H, \tau}(\CC) &= \BB^{\times} \backslash \mathfrak{H}^{\pm} \times \widehat{\BB}^{\times}/H \times \lbrace \operatorname{cusps} \rbrace \end{align*} 
in such a way that $Q_H$ is represented by $[z_0, 1]$, where $z_0 \in \mathfrak{H}$ is the unique fixed point of $K^{\times}$ via the action induced by the embedding 
$K \rightarrow \BB$. Note that for any vector $\varphi \in \pi_A$, we obtain a well-defined point $\varphi(Q) \in A(K^{\ab})$.

\subsubsection{Automorphic periods}

Fix a ring class character $\Omega: \AK /K^{\times} \longrightarrow {\CC}^{\times}$ taking values in some subfield $L$ of 
$\overline{\QQ}$ containing the field $R = \End_F(A) \otimes_{\ZZ} {\QQ} $. 
Recall that given a simple abelian variety $A/F$ parametrized by a Shimura curve 
$M = \lbrace M_H \rbrace_H$ over $F$, we have some associated automorphic representation $\pi_A = \otimes_v \pi_{A, v}$ of $\widehat{\BB}^{\times}$. 
We shall assume (as throughout) that the central character of $\pi_A$ is trivial.
We then define for any choice of vector $\varphi \in \pi_A$ the period integral 
\begin{align}\label{AP-1} P_{\Omega}^{\BB}(\varphi) &= \int_{\AK^{\times} / K^{\times}}
\varphi(Q^t) \otimes_R \Omega(t) dt  \in A(K^{\ab}) \otimes_R L.\end{align} 
Here, $dt$ is the Haar measure on $\AK^{\times}/K^{\times}$ having total volume $1$. 
An analogous definition is of course be made for the dual representation $\pi_{A^{\vee}}$. It is also the case that we have the inclusion 
\begin{align*}P_{\Omega}^{\BB}(\varphi) \in A(\Omega) := \left( A(K^{\ab}) \otimes_R L_{\Omega} \right)^{G_K^{\ab}}, \end{align*} 
where $L_{\Omega}$ denotes the $R$-vector space $L$ with action of $G_K^{\ab}$ given by the character $\Omega$, 
and that the correspondence $\varphi \mapsto P_{\Omega}^{\BB}(\varphi)$ defines an element of 
\begin{align*} \Hom_{\AK^{\times}}(\pi_A \otimes \Omega, L) \otimes_L A(\Omega) \end{align*} 
We refer again to the discussion in \cite[$\S 1.3$]{YZZ} for more details.

\subsubsection{Main formula}

The theorem of Tunnel \cite{Tu} and Saito \cite{Sa} (cf.~\cite[Theorem 1.3]{YZZ}) implies that the space $\Hom_{\AK^{\times}}(\pi_A \otimes \Omega, L)$ 
is at most $1$-dimensional, with dimension $1$ if and only if the set $\Ram(\BB)$ of places of $F$ where $\BB$ ramifies is given by
\begin{align*} \Ram(\BB) &= \lbrace v \in V_F: \epsilon(1/2, \pi_{A, v} \times \Omega_v) \neq \Omega_v(-1)\eta_v(-1)\rbrace .\end{align*} 
Recall that the global root number $\epsilon(1/2, \pi_A \times \Omega)$ in this setting must be equal to $-1$, 
in which case the corresponding central value $L(1/2, \pi_A \times \Omega)$ must vanish.

Let us now commit an abuse of notation in writing $\pi = \otimes_v \pi_v$ to denote the automorphic representation $\pi_A = \Hom^0(J, A)$ of 
$\widehat{\BB}^{\times}$ introduced above. Let us also assume that the space $\Hom_{\AK^{\times}}(\pi_A \otimes \Omega, L)$ has dimension $1$. 
The main result of \cite{YZZ} relates the value $L'(1/2, \pi \times \Omega)$ to a certain generator of this space $\Hom_{\AK^{\times}}(\pi \otimes \Omega, L)$. 
To be more precise, let 
\begin{align*} \mathcal{P}(\pi, \Omega) &= \Hom_{\AK^{\times}}(\pi \otimes \Omega, L). \end{align*} 
Let \begin{align*} \mathcal{P}(\pi_v, \Omega_v) &= \Hom_{K_v^{\times}}(\pi_v \otimes 
\Omega_v, L) \end{align*} for each place $v$ of $F$, whence there is a decompositions 
$\mathcal{P}(\pi, \Omega) = \otimes_v \mathcal{P}(\pi_v, \Omega_v)$, 
as well as an analogous decomposition $\mathcal{P}(\pi^{\vee}, \Omega^{-1}) 
= \otimes_v \mathcal{P}(\pi_v^{\vee}, \Omega_v^{-1})$ for the associated contragredient representations. 
The main idea of \cite{YZZ} is to find an explicit generator 
\begin{align*} \alpha 
&= \otimes_v \alpha_v \in \mathcal{P}(\pi, \Omega) \otimes \mathcal{P}(\pi^{\vee}, \Omega^{-1})
= \otimes_v \mathcal{P}(\pi_v, \Omega_v) \otimes \mathcal{P}(\pi_v^{\vee}, \Omega_v^{-1}), \end{align*} 
and then to use various calculations with the theta correspondence in the style of Waldspurger \cite{Wa} 
(with their geometric description of the Jacquet-Langlands lift) to relate the
this generator to the central derivative values $L'(1/2,  \pi \times \Omega)$.
The generator $\alpha = \otimes \alpha_v$ they obtain
is defined formally as follows. For each place $v$ of $F$, fix a Haar 
measure $dt_v$ on $K_v^{\times}/F_v^{\times}$ such that (i) the product 
measure over all places $v$ gives the Tamagawa measure on 
$\AK^{\times}/K^{\times}$ and (ii) the maximal compact subgroup
$\mathcal{O}_{K_v}^{\times}/\mathcal{O}_{F_v}^{\times}$ has a volume in
$\QQ$ for each nonarchimedean place $v$. Then, for vectors $\varphi_{1,v} 
\in \pi_v$ and $\varphi_{2,v} \in \pi_v^{\vee}$, the local component $\alpha_v$
is defined formally by \begin{align*} \alpha(\varphi_{1,v}, \varphi_{2, v}) &:= 
\frac{ L(1, \eta_v) L(1, \pi_v, \ad)}{\zeta_{F_v}(2)L(1/2, \pi_v \times \Omega_v)}
\int_{K_v^{\times}/F_v^{\times}} (\pi_v(t) \varphi_{1, v}, \varphi_{2, v})_v 
\Omega_v(t) dt. \end{align*} Fixing an embedding $\iota: L \rightarrow {\CC}$,
this local integral can be seen to take values in $L$ (independently of the choice of 
embedding $\iota$), as explained in \cite[$\S 1.3$]{YZZ}. The local invariant pairings 
$(\cdot, \cdot)_v: \pi_v \times \pi_v^{\vee} \longrightarrow R$ also satisfy the compatibility 
relation $(\cdot, \cdot) = \otimes_v (\cdot, \cdot)_v$, where $(\cdot, \cdot): \pi_A \times \pi_A^{\vee} 
\longrightarrow R$ denotes the perfect $\BB^{\times}(\AF)$-invariant pairing introduced above. 
The following formula, which is the main result of the \cite{YZZ}, relates this element 
$\alpha = \otimes \alpha_v$ to the central derivative values $L'(1/2, \pi_A \times \Omega)$.

\begin{theorem}\label{YZZ} 

Given decomposable vectors $\varphi_1 \in \pi_A$ and $\varphi_2 \in \pi_A^{\vee}$ in the setup above, we have
the following identity in $L \otimes_{\QQ} \CC$ : \begin{align}\label{YZZMF} \langle P_{\Omega}^{\BB}(\varphi_1), 
P^{\BB}_{\Omega^{-1}}(\varphi_2) \rangle_L &= \frac{\zeta(2)L'(1/2, \pi_A \times \Omega)}{4 L(1, \eta)^2 L(1, \pi_A, \ad)}
\cdot \alpha(\varphi_1, \varphi_2). \end{align} \end{theorem}

\begin{proof}See Yuan-Zhang-Zhang, \cite[Theorem 1.2]{YZZ}. \end{proof} 

\section{Distribution relations on quaternion algebras}

We now recall some general distribution relations on split quaternion algebras, following the appendix of \cite{CV}. We shall use these results in our constructions below.

\subsection{Local quadratic orders}

Let $K/F$ be any quadratic extension of number fields. Writing $\mathcal{O}_F$ and $\mathcal{O}_K$ to denote the rings of integers of $F$ and 
$K$ respectively, we consider the localizations $\OFP$ and $\OKP$ of these rings at a fixed prime $\pp$ of $\mathcal{O}_F$.

\begin{lemma}\label{6.1} 
Let $Z$ be any $\OFP$-order in $K_{\pp}$. Then, there exists a unique integer $n = l_{\pp}(Z) \geq 0$ for which $Z = \OFP + \pp^n \OKP$, whence we write $Z = Z_n$. 
\end{lemma}

\begin{proof} The result is standard. See e.g. the proof given in \cite[$\S$ 6.1]{CV}. \end{proof}

\subsection{Distribution relations on split quaternion algebras}

Let us now consider any quaternion algebra $B$ defined over $F$ which (i) contains the quadratic extension $K$ and (ii) splits at the fixed prime $\pp$. 
Let $H$ be any compact open subgroup of $\widehat{B}^{\times}$ of the form $H = H^{\pp} R_{\pp}^{\times}$, where $R_{\pp}$ is an Eichler order of the 
local quaternion algebra $B_{\pp}$ of level $\pp^{\delta}$ for some integer $\delta \geq 0$ . There is a left action of $K_{\pp}^{\times}$ on the quotient 
$B_{\pp}^{\times}/R_{\pp}^{\times}$ given by the rule 
\begin{align*} \gamma \star x &= [\gamma b], ~~~~~\gamma \in K_{\pp}^{\times}, ~ x = [b] \in B_{\pp}^{\times}/R_{\pp}^{\times}.\end{align*}
Given a class $x = [b]$ in $B_{\pp}^{\times}/R_{\pp}^{\times}$, the stabilizer of this action is given by $Z(x)^{\times}$, where $Z(x)$ 
is the $\OFP$-order of $K_{\pp}$ determined by the intersection $Z(x) = K_{\pp} \cap b R_{\pp} b^{-1}$. Using Lemma \ref{6.1} 
above, we assign to this class $x = [b]$ the uniquely determined integer $l_{\pp}(x) = l_{\pp}(Z(x)) \geq 0$, i.e., so that 
$\Stab_{K_{\pp}^{\times}}(x) = Z_{l_{\pp}(x)}^{\times}.$ The main results of \cite[Appendix $\S$ 6]{CV} use \cite[Lemma 6.1]{CV} to relate the trace operator 
\begin{align*} \Tr(x) &= \sum_{\gamma \in Z_{l_{\pp}(x)-1}^{\times}/ Z_{l_{\pp}(x)}^{\times}} \gamma \star  x = \sum_{\gamma \in Z_{l_{\pp}(x)-1}^{\times}/ 
Z_{l_{\pp}(x)}^{\times}} [\gamma b] \end{align*} to the (local) Hecke operator \begin{align*}T_{\pp}(x) &= [R_{\pp}^{\times} \xi_{\pp} R_{\pp}^{\times}](x). \end{align*} 
To describe these relevant results in more detail, we shall divide into cases on the exponent $\delta \geq 0$ in the level of the local Eichler order $R_{\pp}$.

\subsubsection{The case of $\delta = 0$} 

Suppose that $R_{\pp}$ has level $\pp^{\delta} = 1$, whence $R_{\pp}$ is a maximal order in $B_{\pp}$. 
Let $V$ be a simple left $B_{\pp}$-module for which $V \approx F_{\pp}^2$ as an $F_{\pp}$-vector space. 
The embedding $K_{\pp} \longrightarrow B_{\pp}$ endows $V$ with the structure of a free rank one (left) 
$K_{\pp}$-module. Let $\mathcal{L} = \mathcal{L}(V)$ denote the set of  $\OFP$-lattices in $V$. Fix 
a base lattice $L_0 \in \mathcal{L}$ such that $\lbrace 
\alpha \in B_{\pp}: \alpha L_0 \subset L_0 \rbrace = R_{\pp}$.
Thus we may fix a bijection \begin{align}\label{d0}
B_{\pp}^{\times}/R_{\pp}^{\times} &\longrightarrow 
\mathcal{L}, ~~~ b \longmapsto b L_0. \end{align}
The induced left action of $K_{\pp}^{\times}$ on 
$\mathcal{L}$ is given by the rule 
\begin{align*} \gamma \star L &= \gamma L, 
~~~~~ \gamma \in K_{\pp}^{\times}, ~ L = bL_0
\in \mathcal{L}. \end{align*} The induced 
function $l_{\pp}$ on $\mathcal{L}$ is given
by the rule that sends a lattice $L$ to 
the unique integer $l_{\pp}(L) \geq 0$ 
for which $\lbrace \gamma \in 
K_{\pp}^{\times}: \gamma L \subset 
L \rbrace$ equals $Z_{l_{\pp}(L)}$.

\begin{lemma}\label{lattice} The induced function $l_{\pp}$
on $\mathcal{L}$ defines a bijection  \begin{align*} 
K_{\pp}^{\times} \backslash B_{\pp}^{\times}/ 
R_{\pp}^{\times} \approx K_{\pp}^{\times} 
\backslash  \mathcal{L} &\longrightarrow 
{\bf{Z}}_{\geq 0}. \end{align*} \end{lemma}

\begin{proof} See \cite[Lemma 6.2]{CV}. \end{proof}

Let ${\bf{Z}}[\mathcal{L}]$ denote the free abelian group 
generated by $\mathcal{L}$.

\begin{definition} Let $L \in \mathcal{L}$ be a lattice.

\begin{itemize}

\item[(i)] The {\it{lower neighbours}} of $L$ are the 
sublattices $L' \subset L$ for which we have 
$L/L' \approx \OFP/ \pp \OFP$. The {\it{upper neighbours}} 
of $L$ are the superlattices $L' \supset L$ for which we 
have $L'/L\approx \OFP/ \pp \OFP $. 

\item[(ii)] The {\it{lower Hecke operator}} $T_{\pp}^l$
on ${\bf{Z}}[\mathcal{L}]$ is the operator that sends
a lattice $L$ to the sum of its lower neighbours. 
The {\it{upper Hecke operator}} $T_{\pp}^u$
on ${\bf{Z}}[\mathcal{L}]$ is the operator that sends
a lattice $L$ to the sum of its upper neighbours. 

\item[(iii)] Given a lattice $L \in \mathcal{L}$ with
$l_{\pp}(L) \geq 1$, the {\it{lower predecessor}}
of $L$ is defined by the lattice $\pr_l(L) = \pp 
Z_{l_{\pp}(L)-1}L$; the {\it{upper predecessor}}
of $L$ is defined by the lattice $\pr_u(L) = Z_{l_{\pp}(L)-1}L$.

\end{itemize} \end{definition}

\begin{remark}
As explained in \cite[Remark 6.3]{CV},
the lower Hecke operator $T_{\pp}^l$ 
corresponds under the fixed bijection 
$(\ref{d0})$ to the double coset operator 
$[R_{\pp}^{\times} \alpha_{\pp} R_{\pp}^{\times}]$, 
where we write $\alpha_{\pp} \in R_{\pp} \approx \M(\mathcal{O}_{F_{\pp}})$ 
to denote some element of $\nrd(\alpha_{\pp}) 
= \varpi_{\pp}$. Similarly, the upper Hecke operator 
$T_{\pp}^u$ corresponds to the double coset 
operator $[R_{\pp}^{\times} \alpha_{\pp}^{-1} 
R_{\pp}^{\times}]$. Note also that \cite[$\S$ 6]{CV}
treats only the lower Hecke operators $T_{\pp}^l$
for simplicity, the analogous discussion for the upper Hecke operators 
$T_{\pp}^{u}$ being simple to deduce. \end{remark} Let us for simplicity of notation define

\begin{align*} \eta_{\pp} &=
\begin{cases} -1 &\text{if $\pp \mathcal{O}_K = \PP$ is inert}\\
0 &\text{if $\pp \mathcal{O}_K = \PP^2$ is ramified} \\ 
1 &\text{if $\pp \mathcal{O}_K = \PP {\PP}^*$ is split. } \end{cases} 
\end{align*}

\begin{lemma}\label{6.5} Let $L$ be a lattice in $\mathcal{L}$.

\begin{itemize}

\item[(i)] If $l_{\pp}(L) =0$, then there are precisely $1+ \eta_{\pp}$
lower neighbours $L' \subset L$ for which $l_{\pp}(L') = 0$. 
Explicitly, these lower neighbours are given by \begin{align*}
L' &=  \begin{cases} \emptyset &\text{if $\pp \mathcal{O}_K = \PP$ is inert}\\
\PP L &\text{if $\pp \mathcal{O}_K = \PP^2$ is ramified} \\ 
\PP L, {\PP}^*L &\text{if $\pp \mathcal{O}_K = \PP {\PP}^*$ is split. } 
\end{cases} \end{align*}

\item[(ii)] If $l_{\pp}(L) >0$, then there exists a unique lower neighbour
$L' \subset L$ for which $l_{\pp}(L') \leq l_{\pp}(L)$. Explicitly, this lower
neighbour $L' \subset L$ is given by the lower predecessor $L' = \pr_l(L) = 
\pp Z_{l_{\pp}(x)-1}L$, which satisfies $l_{\pp}(L') = l_{\pp}(L) -1$.

\item[(iii)] In either case, the remaining lower neighbours $L' \subset L$
satisfy the property that $l_{\pp}(L') = l_{\pp}(L) +1$ . These remaining
lower neighbours $L'$ are also permuted faithfully and transitively by 
the action of $Z_{l_{\pp}(L)}^{\times}/Z_{l_{\pp}(L) +1}^{\times}$, and 
moreover have $L$ as their common upper predecessor.

\end{itemize} \end{lemma}
 
 \begin{proof} See \cite[Lemma 6.5]{CV}, the 
 result is deduced from \cite[Lemma 6.1]{CV}. \end{proof}

One can deduce from this the following result.

\begin{corollary}\label{6.6} Let $x$ be a class in 
$B_{\pp}^{\times} / R_{\pp}^{\times}$ with $l_{\pp}(x) \geq 1$. 
Then, we have \begin{align*} \Tr(x) &= T_{\pp}^l(x') - x'', \end{align*}
where $x' = \pr_u(x)$, and 

\begin{align*}
x'' &= \begin{cases} 0 &\text{if $l_{\pp}(x) =1$ and $\pp \mathcal{O}_K = \PP$ is inert}\\
\varpi_{\PP} x' &\text{if $l_{\pp}(x) =1$ and 
$\pp \mathcal{O}_K = \PP^2$ is ramified} \\ 
(\varpi_{\PP}+ \varpi_{{\PP}^*})x'
&\text{if $l_{\pp}(x) =1$ and $\pp \mathcal{O}_K = \PP {\PP}^*$ is split} \\
\pr_l(x') &\text{if $l_{\pp}(x) \geq 2$}. \end{cases} \end{align*}
Here, for a prime $\PP$ of $\mathcal{O}_K$, $\varpi_{\PP}$
denotes a uniformizer of $\PP$ (which corresponds 
under the reciprocity map $\rec_K$ to the geometric 
Frobenius at $\PP$). \end{corollary}

\begin{remark}
The discussion and results above do not depend upon
the choice of base lattice $L_0$ in $\mathcal{L}$; see
\cite[$\S$ 6.2]{CV} for more explanation.
\end{remark}

\subsubsection{The case of $\delta = 1$}

Suppose now that $R_{\pp}$ is Eichler of level $\pp^{\delta} = \pp$. Keep $V$ as defined above. 
Let $\mathcal{L}_1 = \mathcal{L}_1(V)$ denote the set $1$-lattices of $\mathcal{L}$, i.e. the
set of pairs of lattices $L = (L(0), L(1))$ with $L(0), L(1) \in \mathcal{L}$ such that $L(1) \subset L(0)$ is a sublattice for which $L(0)/L(1) \approx \OFP/ \pp \OFP$. 
The group $B_{\pp}^{\times} \approx \operatorname{GL}(V)$ acts transitively on $\mathcal{L}_1$, and so
we can fix a $1$-lattice $L_0 = (L_0(0), L_0(1))$ whose stabilizer under this action is given by $R_{\pp}^{\times}$. Fixing such a $1$-lattice, 
we may also fix a bijection 
\begin{align}\label{d1} B_{\pp}^{\times}/R_{\pp}^{\times} &\longrightarrow \mathcal{L}_1, ~~~ b \longmapsto b L_0. \end{align}
We can associate to any $L = (L(0)), L(1)) \in \mathcal{L}_1$ a pair of integers 
$l_{\pp, 0} = l_{\pp}(L(0))$ and $l_{\pp, 1} = l_{\pp}(L(1))$, whence we define $l_{\pp}(L)$ to be the maximum, 
\begin{align*} l_{\pp}(L) &= \max ( l_{\pp, 0}(L), l_{\pp, 1}(L)) = \max( l_{\pp}(L(0)), l_{\pp}(L(1))). \end{align*}
This definition leads to the following possible orientations for a given $1$-lattice $L$.

\begin{definition} Let $L = (L(0), L(1))$ be a $1$-lattice in the set $\mathcal{L}_1$.

\begin{itemize}

\item[(i)] We say that $L$ is of {\it{type I}} 
if $l_{\pp, 0} < l_{\pp, 1}$ (whence $l_{\pp, 1} 
=  l_{\pp, 0}+1$), in which case the {\it{leading 
vertex}} of $L$ is defined to be $L(1)$. If
$l_{\pp}(L) \geq 2$, then we also define
the {\it{predecessor}} of $L$ to be the 
$1$-lattice in $\mathcal{L}_1$ defined by 
\begin{align*} \pr(L) &= (L(0), \pr_l(L(0))) 
= (L(0), \pp Z_{l_{\pp}(L)-2}L(0) ). \end{align*}
 
\item[(ii)] We say that $L$ is of {\it{type II}} 
if $l_{\pp, 0} > l_{\pp, 1}$ (whence $l_{\pp, 0} 
=  l_{\pp, 1}+1$), in which case the {\it{leading 
vertex}} of $L$ is defined to be $L(0)$. If
$l_{\pp}(L) \geq 2$, then we also define
the {\it{predecessor}} of $L$ to be the 
$1$-lattice in $\mathcal{L}_1$ defined by 
\begin{align*} \pr(L) &= (\pr_u(L(1)), L(1)) 
= (Z_{l_{\pp}(L)-2}L(1), L(1)). \end{align*} 

\item[(iii)] We say that $L$ is of {\it{type III}} 
if $l_{\pp, 0} = l_{\pp, 1}$ (whence 
$l_{\pp}(L) = 0$, $\eta_{\pp} \in \lbrace 0, 1 
\rbrace$, and $L(1) = \PP L(0)$ or $L(1) 
= {\PP}^*L(0)$). In this case, as a convention,
we define the {\it{leading vertex}} of $L$ 
to be $L(0)$. \end{itemize} \end{definition}

\begin{remark}
The type of a $1$-lattice $L = (L(0), L(1))$ in $\mathcal{L}_1$ is invariant under the action of $\gamma \in K_{\pp}^{\times}$. 
More precisely, $L$ and $\gamma L = \gamma \star L$ have the same type, and moreover $\pr(\gamma L) = \gamma \pr (L) = \gamma \star \pr(L)$ if $l_{\pp}(L) \geq 2$.
\end{remark}

Let ${\bf{Z}}[\mathcal{L}_1]$ denote the free abelian group generated by $\mathcal{L}_1$. 

\begin{definition} (i) The {\it{lower Hecke operator}} $T_{\pp}^l$ on ${\bf{Z}}[\mathcal{L}_1]$ is the rule that sends a $1$-lattice $L = (L(0), L(1))$ to the sum of
all $1$-lattices $L' = (L'(0), L'(1))$ for which $L'(0) = L(0)$ but $L'(1) \neq L(1)$. (ii) The {\it{upper Hecke operator}} $T_{\pp}^u$ on ${\bf{Z}}[\mathcal{L}_1]$ is
the rule that sends a $1$-lattice $L = (L(0), L(1))$ to the sum of all $1$-lattices $L' = (L'(0), L'(1))$ for which $L'(1) = L(1)$ but $L'(0) \neq L(0)$. \end{definition}

\begin{remark}\label{6.10} These Hecke 
operators correspond to the following double 
coset operators, as explained in 
\cite[Remark 6.10]{CV}. Suppose 
that we write the Eichler order 
$R_{\pp}$ as the intersection 
$R_{\pp} = R(0) \cap R(1)$, 
where for $i \in \lbrace 0, 1 \rbrace$, 
we define \begin{align*} R(i) &= \lbrace 
b \in B_{\pp} : b L_0(i) 
\subset L_0(i) \rbrace. \end{align*} 
Then, the lower Hecke operator 
$T_{\pp}^l$ corresponds to the 
double coset operator $[R_{\pp}^{\times} 
\alpha R_{\pp}^{\times}]$ for any element 
$\alpha \in R(0)^{\times} - R(1)^{\times}$, 
and the upper Hecke operator $T_{\pp}^u$ 
corresponds to the double coset operator 
$[R_{\pp}^{\times} \beta R_{\pp}^{\times}]$ 
for any element $\beta \in R(1)^{\times} - R(0)^{\times}$. 
Also, we have the decompositions 
$R(0)^{\times} = R_{\pp}^{\times} \coprod 
R_{\pp}^{\times} \beta R_{\pp}^{\times}$ 
and $R(1)^{\times} = R_{\pp}^{\times} 
\coprod R_{\pp}^{\times} \alpha R_{\pp}^{\times}$
\end{remark}

\begin{definition} The {\it{$\pp$-new quotient}} 
${\bf{Z}}[\mathcal{L}_1]^{\pp-\new}$ of 
${\bf{Z}}[\mathcal{L}_1]$ is the quotient 
of ${\bf{Z}}[\mathcal{L}_1]$ by the 
${\bf{Z}}$-submodule spanned by
elements of the form \begin{align*}
\sum_{L' = (L'(0), L'(1)) \atop L'(0) = M} L', 
~~~~~ \sum_{L' = (L'(0), L'(1)) \atop L'(1) = M} L', 
\end{align*} where $M$ is some lattice in 
$\mathcal{L}$. Observe that by construction,
we have the congruences \begin{align*}
T_{\pp}^l \equiv T_{\pp}^{u} \equiv -1 
~~\text{ ~on~ ${\bf{Z}}[\mathcal{L}_1]^{\pp-\new}$.}
\end{align*} \end{definition} One can again use 
Lemma \ref{6.5} to relate the trace operator 
\begin{align*} \Tr(L) &= \sum_{ \gamma \in
Z_{l_{\pp}(L)-1}^{\times}/ Z_{l_{\pp}(L)}^{\times}}
\gamma \star L = \sum_{ \gamma \in
Z_{l_{\pp}(L)-1}^{\times}/ Z_{l_{\pp}(L)}^{\times}}
[\gamma b] L_0 \end{align*} in ${\bf{Z}}[\mathcal{L}_1]$
to the operators $T_{\pp}^l(L)$ and $T_{\pp}^{u}(L)$ to obtain the following result.
Here, we extend the notion of types to classes of $B_{\pp}^{\times} / R_{\pp}^{\times}$ in the natural way via $(\ref{d1})$.

\begin{lemma}\label{6.11}
Let $x$ be a class in $B_{\pp}^{\times} / R_{\pp}^{\times}$ 
with $l_{\pp}(x) \geq 2$. Then, we have 
\begin{align*} \Tr(x) &= \begin{cases}
T_{\pp}^l(\pr(x)) &\text{if $x$ is of type I} \\
T_{\pp}^{u}(\pr(x)) &\text{if $x$ is of type II. }
 \end{cases} \end{align*} Moreover, in the 
 $\pp$-new quotient of ${\bf{Z}} [B_{\pp}^{\times} 
 / R_{\pp}^{\times}]$ corresponding to 
 ${\bf{Z}}[\mathcal{L}_1]^{\pp-\new}$,

 \begin{align*} 
\Tr(x) &=- \pr(x). \end{align*} \end{lemma}

\begin{proof} See \cite[Lemma 6.11]{CV}. 
\end{proof}

\begin{remark} The constructions and results above do depend on the choice of base $1$-lattice $L_0 = (L_0(0), L_0(1))$. In particular, the definition
of type depends upon the orientation of the underlying Eichler order $R_{\pp} = R(0) \cap R(1)$. \end{remark}

\section{Construction of $p$-adic interpolation series}

We now construct $p$-adic interpolation series on the profinite group $X$ for the values $L^{(k)}(1/2, \pi \times \Omega)$, i.e. where $\Omega$ is a character factoring through $X$. 
We divide into cases $k$ on the generic root number $\epsilon = \epsilon(1/2, \pi \times \Omega) \in \lbrace \pm 1 \rbrace$. Here, we shall also write $\star$ to denote 
the natural action of $\widehat{K}^{\times}$ by left multiplication on the set $K^{\times} \backslash \widehat{B}^{\times} / \widehat{R}^{\times}$ induced by the choice of embedding 
$K$ into $B$ or $\BB$ (depending on whether $k=0$ or $k=1$ respectively).

\begin{remark} Our constructions below works more generally whenever the vector in known $\varphi \in \pi'$ to take values in a discrete valuation ring and  
each prime dividing the level structure is known to split in $K$. \end{remark}

\subsection{The case of $k=0$} 

Let $\pi = \otimes_v \pi_v$ be a cuspidal automorphic representation of $\GL(\AF)$ having trivial central character. Fix a prime $\pp$ of $\mathcal{O}_F$ 
with underlying rational prime $p$, and an embedding $\overline{\QQ} \longrightarrow \overline{\QQ}_p$. We assume Hypotheses \ref{pi} and \ref{local+1}. 

\begin{lemma}\label{normalization}

Let $L$ be a finite extension of ${\QQ}_p$ containing the eigenvalues of $\pi$. Let $\mathcal{O} = \mathcal{O}_L$ denote its ring of integers. If the quaternion 
algebra $B$ is totally definite, then we can and do choose $\varphi \in \pi'$ in such a way that $\varphi$ takes values in the ring $\mathcal{O}$. \end{lemma}

\begin{proof} By Hypothesis \ref{pi}, the eigenvalues of $\pi = \JL(\pi')$ of $\GL(\AF)$ are algebraic. Hence, ${\QQ}(\pi)$ is a number field, and there exists a finite 
extension $L$ of ${\QQ}_p$ containing the eigenvalues of $\pi$. Any vector $\varphi \in \pi'$ has the same eigenvalues as $\pi$ by Jacquet-Langlands, and 
is determined uniquely up to multiplication by nonzero complex scalar by this condition. Moreover, we can view any $\varphi \in \pi'$ as an automorphic function on 
the finite set $C(B; H) = B^{\times} \backslash \widehat{B}^{\times} /H$, for $H \subset \widehat{B}^{\times}$ some compact open subgroup (with $H_{\pp} = R_{\pp}^{\times}$). 
Fixing a basis for the space of functions $C(B; H) \longrightarrow {\bf{C}}$, we then can choose $\varphi \in \pi'$ to take values in ${\bf{Q}}(\pi)$, or even ${\ZZ}(\pi)$. \end{proof}

We now give the main construction of $p$-adic $L$-functions for the $k=0$ setting. 

\subsubsection{The case of $\delta = 0$}

Let us first consider the case where $R_{\pp}$ of $B_{\pp}$ is Eichler of level $\pp^{\delta} = 1$, i.e. that $R_{\pp}$ is maximal. Recall that we let $V$ be a simple left 
$B_{\pp}$-module for which $V \approx F_{\pp}^2$ as an $F_{\pp}$-vector space. Recall as well that we write $\mathcal{L} = \mathcal{L}(V)$ 
to denote the set of $\OFP$-lattices in $V$, and that we fix a lattice $L_0 \in \mathcal{L}$ in such a way that $\lbrace \alpha \in B_{\pp}: 
\alpha L_0 \subset L_0 \rbrace = R_{\pp}$. Thus we may fix a bijection \begin{align} \label{b0} B_{\pp}^{\times}/R_{\pp}^{\times} &\longrightarrow \mathcal{L}, ~~~
x \longmapsto x L_0. \end{align} Using this bijection, we now fix the following sequence of local classes $x_n \in B_{\pp}^{\times}/R_{\pp}^{\times}$. 

\begin{definition}\label{sc0} Fix a $K_{\pp}$-basis $e$ of $V$. For each integer $n \geq 0$, let $L_n = L_n(e)$ denote the lattice in $\mathcal{L}$ defined by $L_n = Z_n e$.
We then define $x_n = x_n(e)$ to be the class in $B_{\pp}^{\times}/R_{\pp}^{\times}$ corresponding to $L_n$ under the fixed bijection $(\ref{b0})$. \end{definition}

\begin{lemma}\label{p0} The sequence $x = x(e)$ of classes $(x_n)_{n \geq 0}= (x_n(e))_{n \geq 0}$ in $B_{\pp}^{\times}/R_{\pp}^{\times}$ of Definition 
\ref{sc0} satisfies the following properties for each integer $n \geq 0$. \begin{itemize} 

\item[(i)] $l_{\pp}(x_n) = n$.

\item[(ii)] $x_n = \pr_u(x_{n+1})$.

\end{itemize} \end{lemma}

\begin{proof} Property (i) is a direct consequence of the definition of $l_{\pp}(x_n) = l_{\pp}(L_n)$. Property (ii) is then a direct consequence of the definition
of upper predecessor $\pr_u(x_n)$, i.e. as $\pr_u(x_{n+1}) = \pr_u(L_{n+1}) = Z_n L_{n+1} = Z_n Z_{n+1} e = Z_n e = L_n$. \end{proof}

\begin{remark}  Given such a sequence of local classes $(x_n)_{n \geq 0}$ in $B_{\pp}^{\times}/R_{\pp}^{\times}$, we shall also write each $x_n$ to
denote its corresponding adele in $\widehat{B}^{\times}/ \widehat{R}^{\times}$.\end{remark}

Let us now fix a sequence of classes $(x_n)_{n \geq 0}$ as in Definition \ref{sc0} above. Observe that Corollary \ref{6.6} 
above implies that for each integer $n \geq 0$, we have the relation 
\begin{align*} \Tr(x_{n+1}) &= T_{\pp}^l(x_n) - x_{n}^*\end{align*} in ${\bf{Z}} [B_{\pp}^{\times}/ R_{\pp}^{\times}]$, where for $\varpi_{\PP}$ 
a fixed uniformizer at $\PP$, \begin{align*} x_{n}^* &= \begin{cases} 0 
&\text{if $n =0$ and $\pp \mathcal{O}_K = \PP$ is inert}\\ \varpi_{\PP} x_n &\text{if $n =0$ and $\pp \mathcal{O}_K = \PP^2$ is ramified} \\ 
(\varpi_{\PP} + \varpi_{{\PP}^*} )x_n &\text{if $n =0$ and $\pp \mathcal{O}_K = \PP {\PP}^*$ is split} \\ \pr_l(x_n) &\text{if $n \geq 1$}. \end{cases} \end{align*} 

\begin{corollary}\label{6.5C} Assume that $\pi = \JL(\pi')$ has trivial central character, and let $\varphi \in \pi'$ be any decomposable 
vector whose component at $\pp$ is fixed by $R_{\pp}^{\times}$. Let $(x_n)_{n \geq 0}$ be the sequence of classes of Lemma \ref{sc0}. 
Then, for each integer $n \geq 1$, \begin{align*} \Tr \left( \varphi (x_{n+1}) \right) &:= \sum_{\gamma \in Z_n^{\times} /Z_{n+1}^{\times}} 
\varphi (\gamma \star x_{n+1}) = T_{\pp}^l \varphi(x_n) - \varphi(x_{n-1}). \end{align*} \end{corollary}

\begin{proof} Extending by linearity, we obtain from Corollary \ref{6.5} the relation 
\begin{align*} \Tr \left( \varphi (x_{n+1}) \right) &:= \sum_{\gamma \in Z_n^{\times} /Z_{n+1}^{\times}} 
\varphi (\gamma \star x_{n+1}) = T_{\pp}^l \varphi(x_n) - \varphi(\pr_l(x_{n}))
\end{align*} for each integer $n \geq 1$. Here, we have viewed each 
local class $x_n \in B_{\pp}^{\times}/R_{\pp}^{\times}$ as its corresponding 
adele in $\widehat{B}^{\times}/\widehat{R}^{\times}$. We have also used the 
fact that the operators $\Tr$ and $T_{\pp}^l$ affect only the 
component at $\pp$. Now, recall that we write $L_n = x_n L_0$ to 
denote the lattice in $\mathcal{L}$ corresponding to $x_n$, whence
the lower predecessor $\pr_l(x_n)$ is by definition the class corresponding
under $(\ref{b0})$ to the lattice $\pr_l(L_n) = \pp Z_{n-1} L_n$. Observe that 
$\pr_l(L_n) = \pp \pr_u(L_n)$, whence we argue that $\varphi(\pr_l(x_n)) 
= \varphi(\pr_u(x_n))$ as a consequence of the fact that $\varphi$ has trivial 
central character. The desired relation thens follow from the property 
$\pr_u(x_n) = x_{n-1}$. \end{proof}

\begin{definition}\label{dist0} Assume that $\pi = \JL(\pi')$ is $\pp$-ordinary, 
with trivial central character. Let $\varphi \in \pi'$ be a decomposable vector
whose local component at $\pp$ is fixed by $R_{\pp}^{\times}$, which we can 
and do normalize to take values in the ring $\mathcal{O}$. Writing $a_{\pp} = 
a_{\pp}(\pi)$ to denote the $T_{\pp}^l$-eigenvalue of $\pi$, and $q_{\pp}$ the 
cardinality of the residue field at $\pp$, let $\alpha_{\pp} = \alpha_{\pp}(\pi)$ 
denote the $p$-adic unit root of the Hecke polynomial $t^2 - a_{\pp}t + q_{\pp}$.
We then define the system $\lbrace \vartheta_n \rbrace_{n \geq 1} = 
\lbrace \vartheta_n(x_n, \varphi) \rbrace_{n \geq 1}$ to
be the sequence of mappings $\vartheta_n: X_n \longrightarrow \mathcal{O}$
given by the rule \begin{align*} \vartheta_n:X_n \longrightarrow \mathcal{O}, ~~~
A \longmapsto \alpha_{\pp}^{1-n} \cdot \varphi(A \star x_n) - \alpha_{\pp}^{-n} 
\cdot \varphi(A \star x_{n-1} ). \end{align*} \end{definition}

Before we continue, let us explain briefly that the points $x_n$ here are well-defined.
Recall that we fix a maximal order $R \subset B$ and a function $\varphi$ on the finite set 
$C(\widehat{R}^{\times}) = B^{\times} \backslash \widehat{B}^{\times}/\widehat{R}^{\times}$.
Recall as well that we fix an embedding $K \rightarrow B$, a priori without any special conditions. Observe that this choice of embedding induces a natural 
map from the set $Y(\widehat{R}^{\times}) = K^{\times} \backslash \widehat{B}^{\times}/ \widehat{R}^{\times}$ to $C(\widehat{R}^{\times})$, and moreover that
we may view $\varphi$ as a function on $Y(\widehat{R}^{\times})$ after composition with this natural map $Y(\widehat{R}^{\times}) \rightarrow C(\widehat{R}^{\times})$.
This setup gives rise to a natural definition of {\it{conductor}} of a class or ``point" $x = [b]$ in the set $Y(\widehat{R}^{\times})$. 
It is easy to see that the points $x_n = [b_n]$ of our chosen sequence in Definition \ref{sc0} each have conductor $\mathfrak{c}(x_n)$ 
equal to $\cc \pp^n$, where $\cc$ is determined uniquely by the 
prime-to-$\pp$-part of the conductor of the order $R \cap K $ of $K$. Hence, the values $\varphi(A \star x_n)$ and $\varphi(A \star x_{n-1})$ are only well-defined if 
$A \in \Pic (\mathcal{O}_{\cc \pp^n})$. To proceed, we could either keep track of the $\cc$ explicitly throughout (at the expense of simplicity), 
or else we choose a specific embedding $K \rightarrow B$ to ensure that $\cc = 1$. Let us for simplicity choose the latter option for the rest of this work, 
which we record as follows. 

\begin{corollary} Suppose (as we can) that we fix an embedding $K \rightarrow B$ for which the order $K \cap B$ has $\pp$-power conductor, in fact that 
$K \cap B = \mathcal{O}_K$. Then, each point $x_n = [b_n]$ has conductor $\cc(x_n) = \pp^n$.
\end{corollary}

\begin{proof} That such an embedding exists is known; the local criteria are described in \cite[Ch. II, $\S 3$]{Vi}. In brief, such an embedding exists
in our setting if each prime dividing $\pp^{\delta} \NN^+$ is split in $K$, which is the case by Hypothesis \ref{pi}. \end{proof}

\begin{proposition}\label{DR0} The sequence of mappings $\vartheta_n: X_n \longrightarrow \mathcal{O}$ of Definition \ref{dist0} above forms a distribution on the profinite 
group $X = \varprojlim_n X_n$, and hence an $\mathcal{O}$-valued measure on $X$. \end{proposition}

\begin{proof} It will suffice to show that for each sufficiently large integer $n \geq 1$ and each class $A \in X_n$, the distribution relation \begin{align*}
\vartheta_n(A) &= \sum_{B \in X_{n+1} \atop \pi_{n+1, n}(B) =A} \vartheta_{n+1}(B) \end{align*} holds, where $\pi_{n+1, n}: X_{n+1} \longrightarrow X_n$ denotes
the natural surjective homomorphism. Equivalently, it will suffice under the same conditions to show the relation 
\begin{align*} \vartheta_n(A) &= \sum_{C \in \ker(\pi_{n+1, n})} \vartheta_{n+1}(CA),\end{align*} where $A$ on the right hand side denotes any lift of $A$ to $X_{n+1}$.
Now, by definitions, \begin{align*} \ker(\pi_{n+1, n}) &= \mathcal{O}_{\pp^n, \pp}^{\times}/ \mathcal{O}_{\pp^{n}}^{\times}\mathcal{O}_{\pp^{n+1}, \pp}^{\times}.\end{align*} 
If $n$ is sufficiently large, then $\mathcal{O}_{\pp^n}^{\times} = \mathcal{O}_F^{\times}$ is contained in the local unit group $\mathcal{O}_{\pp^{n+1}, \pp}^{\times}$ (cf. the 
discussion in the proof of of \cite[Lemma 2.8]{CV}). We thus obtain for $n$ sufficiently large that
$\ker(\pi_{n+1, n}) = \mathcal{O}_{\pp^n, \pp}^{\times}/ \mathcal{O}_{\pp^{n+1}, \pp}^{\times} = Z_n^{\times}/Z_{n+1}^{\times}$, whence it will suffice to show 
\begin{align}\label{mainreln} \vartheta_n(A) &= \sum_{\gamma \in Z_n^{\times}/Z_{n+1}^{\times}} \vartheta_{n+1}(\gamma A).\end{align} To show this relation holds for $n$ 
sufficiently large, we evaluate the right hand side of $(\ref{mainreln})$. More precisely, for $n$ sufficiently large, we have by definition of $\vartheta_{n+1}$ that
\begin{align*} \sum_{\gamma \in Z_n^{\times}/Z_{n+1}^{\times}} \vartheta_{n+1}(\gamma A) 
&= \sum_{\gamma \in Z_n^{\times}/ Z_{n+1}^{\times}} \alpha_{\pp}^{-n}\varphi(\gamma A \star x_{n+1}) - \alpha_{\pp}^{-(n+1)} \varphi(\gamma A \star x_{n}) \\ 
&= \alpha_{\pp}^{-n} \sum_{\gamma \in Z_n^{\times}/Z_{n+1}^{\times}} \varphi(\gamma A \star x_{n+1}) - q_{\pp} \alpha_{\pp}^{-(n+1)} \varphi(A \star x_{n}). \end{align*}
Extending by linearity, the result of Corollary \ref{6.6} allows us to identify the right hand side of the last expression with the sum 
\begin{align*} \alpha_{\pp}^{-n} \left( T_{\pp}^l \varphi(A \star x_n) - \varphi(A \star x_{n-1}) \right) - q_{\pp} \alpha_{\pp}^{-(n+1)} \varphi(A \star x_{n}). \end{align*} 
Here, we have used implicitly the easy-to-check fact that the map $L \mapsto \pr_u(L)$ commutes with the action of $K_{\pp}^{\times}$ (cf. \cite[Lemma 6.5]{CV}).
Thus rearranging terms, we have shown (for $n \geq 1$) that \begin{align*} \sum_{\gamma \in Z_n^{\times}/Z_{n+1}^{\times}} \vartheta_{n+1}(\gamma A) 
&= (\alpha_{\pp}^{-n}a_{\pp} - q_{\pp}\alpha_{\pp}^{-(n+1)}) \varphi(A \star x_n) - \alpha_{\pp}^{-n} \varphi(A \star x_{n-1}). \end{align*} Observe now that we have 
the elementary identity \begin{align*}  (\alpha_{\pp}^{-n}a_{\pp} - q_{\pp} \alpha_{\pp}^{-(n+1)}) &= \alpha_{\pp}^{1-n}, \end{align*} i.e. using the constraints imposed 
by the factorization \begin{align*} X^2 - a_{\pp}X + q_{\pp} &= (X - \alpha_{\pp})(X-\beta_{\pp}). \end{align*} Thus for $n$ sufficiently large, we have shown the required relation
\begin{align*} \sum_{\gamma \in Z_n^{\times}/Z_{n+1}^{\times}} \varphi_{n+1}(\gamma A) &= \alpha_{\pp}^{1-n} \varphi(A \star x_n) - \alpha^{-n} \varphi(A \star x_{n-1}) 
= \vartheta_n(A). \end{align*} \end{proof}

\subsubsection{The case of $\delta =1$}

Let us now suppose that the local order $R_{\pp}$ of $B_{\pp}$ is Eichler of level $\pp^{\delta} = \pp$. Recall that we write $\mathcal{L}_1 = \mathcal{L}_1(V)$
to denote the set of $1$-lattices of $\mathcal{L} = \mathcal{L}(V)$. Let us fix a $1$-lattice $L_0 = (L_0(0), L_0(1))$ whose stabilizer under the transitive action of 
$B_{\pp}^{\times}$ is equal to $R_{\pp}^{\times}$. We then fix a bijection \begin{align}\label{b1} B_{\pp}^{\times}/R_{\pp}^{\times}
&\longrightarrow \mathcal{L}_1, ~~~x \longmapsto x L_0. \end{align} Using this bijection, we now fix a sequence of classes $x_n \in B_{\pp}^{\times}/R_{\pp}^{\times}$ as follows.

\begin{definition}\label{sc1} Fix a $K_{\pp}$-basis $e$ of $V$. For each integer $n \geq 0$, let us write $L_n = L_n(e)$ to denote the lattice in $\mathcal{L} = \mathcal{L}(V)$ 
defined by $L_n = \pp^{- \lfloor \frac{n}{2} \rfloor } Z_n e$, where \begin{align*} \lfloor x \rfloor &= \max \lbrace m \in {\bf{Z}} :   m \leq x \rbrace \end{align*}
is the standard floor function. For each $n \geq 1$, let us write $M_n = M_n(e)$ to denote the $1$-lattice in $\mathcal{L}_1 = \mathcal{L}_1(V)$ 
defined by \begin{align*} M_n &= \begin{cases} (L_{n-1}, L_n) &\text{ if $n \equiv 1 \mod 2$} \\ (L_n, L_{n-1}) &\text{ if $n \equiv 0 \mod 2$}\end{cases} \end{align*} 
Let $x_n = x_n(e)$ be the class in $B_{\pp}^{\times}/ R_{\pp}^{\times}$ corresponding to $M_n$ under $(\ref{b1})$. \end{definition}

\begin{lemma}\label{p1} The sequence $x = x(e)$ of classes $(x_n)_{n \geq 1}= (x_n(e))_{n \geq 1}$ in $B_{\pp}^{\times}/R_{\pp}^{\times}$ of Definition \ref{sc1} above 
satisfies the following properties for each integer $n \geq 1$. \begin{itemize} 

\item[(i)] $l_{\pp}(x_n) = n$.

\item[(ii)] $\pr(x_{n+1}) = x_n$. \end{itemize}

\end{lemma}

\begin{proof} Property (i) is a direct consequence of the definition of $l_{\pp}(x_n)$, i.e. as we have $l_{\pp}(x_n) = l_{\pp}(M_n) = \max (l_{\pp,0}(M_{n}), l_{\pp,1}(M_n))= n$. 
Property (ii) is then seen by a direct calculation. To be more precise, since $M_n$ is of type I if $n$ is odd and of type II if $n$ is even, we have by definition that 
\begin{align*}\pr(x_{n+1}) = \pr(M_{n+1}) &= \begin{cases} (L_n, \pr_l(L_n)) = (L_n, \pp Z_{n-1} L_n) &\text{if $n \equiv 0 \mod 2$} \\ (\pr_u(L_n), L_n) = (Z_{n-1}L_n, L_n) 
&\text{if $n \equiv 1 \mod 2$}.\end{cases} \end{align*} Suppose first that the index $n$ is even. We then compute the lower predecessor 
$\pr_l(L_n) = \pp Z_{n-1}L_n = \pp^{1 - \lfloor \frac{n}{2} \rfloor }Z_{n-1}e  = \pp^{- \lfloor \frac{n-1}{2} \rfloor } Z_{n-1} e = L_{n-1}$ to verify the relation $\pr(M_{n+1}) = M_n$, 
as required. Suppose now that the index $n$ is odd. We then compute the lower predecessor $\pr_u(L_n) =  Z_{n-1} L_n = \pp^{- \lfloor \frac{n}{2} \rfloor} Z_{n-1} e 
= \pp^{- \lfloor \frac{n-1}{2} \rfloor } Z_{n-1}e = L_{n-1}$ to verify the relation $\pr(M_{n+1}) = M_n$, as required. \end{proof}

Fix a sequence of classes $(x_n)_{n \geq 1}$ as in Definition \ref{sc1}. We obtain from Lemma \ref{6.11} above the following result.

\begin{corollary}\label{r1} Let $\varphi \in \pi'$ be a decomposable vector whose local component at $\pp$ is fixed by $R_{\pp}^{\times}$, which we can and do normalize to
take values in $\mathcal{O}$. Then, for each integer $n \geq 1$, 
\begin{align*} \Tr \left( \varphi (x_{n+1}) \right) &:= \sum_{\gamma \in Z_{n}^{\times}/Z_{n+1}^{\times}}
\varphi(\gamma \star x_{n+1}) = \begin{cases} T_{\pp}^l \varphi(x_n) &\text{ if $n \equiv 0 \mod 2$}\\ T_{\pp}^{u} 
\varphi(x_n) &\text{ if $n \equiv 1 \mod 2$}. \end{cases} \end{align*} 

\end{corollary}

\begin{proof} Given the sequence of classes $(x_n)_{n \geq 1}$ of Definition \ref{sc1}, along with the properties of Lemma \ref{p1}, Lemma \ref{6.11} implies that for each $n \geq 1$, 
the relation\begin{align*} \Tr  (x_{n+1}) &= \sum_{\gamma \in Z_{n}^{\times}/Z_{n+1}^{\times}} \gamma \star x_{n+1} = \begin{cases}T_{\pp}^l (x_n) 
&\text{ if $n \equiv 0 \mod 2$} \\ T_{\pp}^{u}(x_n) &\text{ if $n \equiv 1 \mod 2$} \end{cases}\end{align*} holds in ${\bf{Z}}[B_{\pp}^{\times}/R_{\pp}^{\times}]$.
Viewing each class $x_n$ as its corresponding adele in $\widehat{B}^{\times}/\widehat{R}^{\times}$, and extending by linearity, we obtain for each 
$n \geq 1$ the relation \begin{align*} \Tr \left( \varphi (x_{n+1}) \right) &= \sum_{\gamma \in Z_{n}^{\times}/Z_{n+1}^{\times}}
\varphi(\gamma \star x_{n+1}) = \begin{cases} T_{\pp}^l \varphi(x_n) &\text{ if $n \equiv 0 \mod 2$}\\ T_{\pp}^{u} 
\varphi(x_n) &\text{ if $n \equiv 1 \mod 2$}.\end{cases} \end{align*} Here, we have used the fact that each of the operators $\Tr$, $T_{\pp}^l$, and $T_{\pp}^{u}$ 
affects only the components at $\pp$. The result follows. \end{proof} 

\begin{definition}\label{dist1} Let $\varphi \in \pi'$ be a decomposable vector whose local component at $\pp$ is fixed by $R_{\pp}^{\times}$ and which we 
normalize to take values in $\mathcal{O}$. Assume that $\varphi$ is a $\pp$-ordinary eigenform for both $T_{\pp}^{u}$ and $T_{\pp}^{l}$ with common (unit) 
eigenvalue $\alpha_{\pp}$. Fix a sequence of classes $(x_n)_{n \geq 1}$ as for Definition \ref{sc0} above. Let $\lbrace \vartheta_n \rbrace_{n \geq 1} = 
\lbrace \vartheta_n(\Phi, x_n) \rbrace_{n \geq 1}$ be the sequence of mappings $\vartheta_n: X_n \longrightarrow \mathcal{O}$ defined by 
\begin{align*} \vartheta_n: X_n &\longrightarrow \mathcal{O}, ~~~ A \longmapsto \alpha_{\pp}^{-n} \cdot \varphi(A \star x_n). \end{align*} \end{definition}

\begin{proposition}\label{DR1}
The sequence of mappings $\vartheta_n: X_n \longrightarrow \mathcal{O}$
of Definition \ref{dist1} defines a distribution on the profinite group 
$X = \varprojlim_n X_n$, and hence an $\mathcal{O}$-valued measure on $X$. \end{proposition}

\begin{proof} As explained in the proof of Proposition \ref{DR0} above, it will suffice to show for each sufficiently large integer $n$ 
that \begin{align*} \vartheta_n(A) &= \sum_{\gamma \in Z_n^{\times}/ Z_{n+1}^{\times}} \vartheta_{n+1}(\gamma A).\end{align*} 
Thus, for the mappings $\lbrace \vartheta_n \rbrace_{n \geq 1}$ of Definition \ref{dist1}, it will suffice to show for each sufficiently 
large integer $n$ that \begin{align*} \alpha_{\pp}^{-n} \cdot \varphi(A \star x_n) &= \sum_{\gamma \in Z_n^{\times}/
Z_{n+1}^{\times}} \alpha_{\pp}^{-(n+1)} \cdot \varphi(\gamma A \star x_{n+1}), \end{align*}
which after multiplying out by $\alpha_{\pp}^{n+1}$ is the same as \begin{align*} \alpha_{\pp} \cdot \varphi(A \star x_n) 
&=  \sum_{\gamma \in Z_n^{\times}/Z_{n+1}^{\times}} \varphi(\gamma A \star x_{n+1}). \end{align*} Now, by
our hypotheses on $\varphi$, we have that $\alpha_{\pp} \cdot \varphi (A \star x_{n}) = T_{\pp} \varphi( A \star x_n)$ for $T_{\pp}$
denoting either operator $T_{\pp}^l$ or $T_{\pp}^{u}$, and so the required relation is a direct consequence of Corollary \ref{r1} above. \end{proof}

\subsubsection{Definitions of $p$-adic $L$-functions}

Recall that in either case on the exponent $\delta \in \lbrace 0, 1 \rbrace$ in the level $\pp^{\delta}$ of $R_{\pp} \subset B_{\pp}$, we construct a sequence 
$x(e) = (x_n(e))_{n \geq \delta}$ of classes $x_n \in B_{\pp}^{\times}/R_{\pp}^{\times}$ according to Definitions \ref{sc0} and \ref{sc1}. These sequences depend 
on the choice of a fixed $K_{\pp}$-basis $e$ of the $B_{\pp}$-module $V$ in the following way. 

\begin{lemma}\label{choiceofe} 
Suppose in either subcase on $\delta \in \lbrace 0, 1 \rbrace$ that we choose a different $K_{\pp}$-basis $e'$ of $V$.
Then, for some element $\sigma = (\sigma_n)_n$ of $X = \varprojlim_n X_n$, we have
\begin{align*} d \vartheta^{(\delta)}(\varphi, x(e')) = d \vartheta^{(\delta)} (\varphi, \sigma \star x(e)) = ( \vartheta_n^{(\delta)} (\varphi, \sigma_n  \star x_n(e)) )_{n \geq 1}, \end{align*} 
equivalently \begin{align*} \theta_{\varphi}^{(\delta)}(x(e')) &= \sigma \theta_{\varphi}^{(\delta)}(x(e)) = (\theta_{\varphi}^{(\delta)}(\sigma_n \star x_n(e))) 
\text{~~ in ~$\mathcal{O}[[X]]$.} \end{align*} \end{lemma}

\begin{proof} Observe that $K_{\pp}^{\times}$ acts simply transitively on the set of $K_{\pp}$-bases of $V$. It follows that there exists a $\gamma \in K_{\pp}^{\times}$
for which $e' = \gamma e = \gamma \star e$. We claim that this action commutes with each of the predecessor operations $\pr_*$ defined above, i.e. that 
$\pr_*( \gamma \star x) = \gamma \star \pr_*(x)$ for each $\gamma \in K_{\pp}^{\times}$ and $x \in B_{\pp}^{\times}/ R_{\pp}^{\times}$, as a direct consequence of the definitions. 
It is then easy to see that for each integer $n \geq 1$, there exists an element $\sigma_n \in X_n$ for which $\vartheta_n^{(\delta)}(\varphi, x_n(e')) 
= \vartheta_n^{(\delta)}(\varphi, \sigma_n \star x_n(e))$. To be more precise, we claim that this $\sigma_n$ is determined by the image of $\gamma$ in $X_n$. Since the 
$\sigma_n$ come from one element $\gamma \in K_{\pp}^{\times}$, it is clear that $(\sigma_n)_n$ defines a compatible sequence, whence 
$\sigma = (\sigma_n)_n$ defines an element of the profinite limit $X = \varprojlim_n X_n$. \end{proof}

Thus, our elements $\theta_{\Phi}^{(\delta)}(e)$ are only well-defined up to multiplication by $\sigma \in X$. To correct this, we make the following modification. 
Let $\Lambda = \mathcal{O}[[X]]$. Given $\lambda \in \Lambda$, let $\lambda^*$ denote the image of $\lambda$ under the involution of $\Lambda$ 
induced by the operation sending group elements $\sigma \in X$ to their inverses $\sigma^{-1} \in X$. The elements 
$\theta_{\varphi}^{(\delta)}(e) \theta_{\varphi}^{(\delta)}(e)^* \in \Lambda$ are then well-defined. This leads us to make the following

\begin{definition} Let $\varphi \in \pi'$ be a decomposable, $\mathcal{O}$-valued, 
$\pp$-ordinary eigenvector as described above, i.e. so that the component of 
$\varphi$ at $\pp$ is fixed by $R_{\pp}^{\times}$, where $R_{\pp} \subset B_{\pp}$
is an Eichler order of level $\pp^{\delta}$ with $\delta \in \lbrace 0, 1 \rbrace$. 
Let $e$ be any $K_{\pp}$-basis of $V$, and $x(e) = (x_n(e))_{n \geq \delta}$ 
the sequence of classes described either by Definition \ref{sc0} if $\delta = 0$
or else by Definition \ref{sc1} if $\delta = 1$. We then define the 
{\it{imprimitive $p$-adic $L$-function}} associated to $\varphi$ and the profinite group $X$
to be the product \begin{align*} L_{\pp}(\varphi, K) &=  L_{\pp}^{(\delta)}(\varphi, K) = 
\theta_{\varphi}^{(\delta)}(x(e)) \theta_{\varphi}^{(\delta)}(x(e))^*\text{~~in ~$\Lambda = 
\mathcal{O}[[X]]$.} \end{align*} \end{definition}

\subsubsection{Interpolation}

We now derive the interpolation properties of the imprimitive $p$-adic $L$-functions $L_{\pp}(\varphi, K)=  L_{\pp}^{(\delta)}(\varphi, K) \in \Lambda = \mathcal{O}[[X]]$. 
This leads us to choose a certain normalization, i.e. to define primitive $p$-adic $L$-functions $\LLL_{\pp}^{(\delta)}(\pi', K)$ which do not depend on our choice of vector 
$\varphi \in \pi'$. Let $\Omega$ be a character of $X$. We view such a character as a homomorphism $\Omega: X \longrightarrow \overline{\bf{Q}}_p$ via our fixed embedding 
$\overline{\bf{Q}} \rightarrow \overline{\bf{Q}}_p$, and moreover (enlarging $\mathcal{O}$ if necessary) as a homomorphism $\Omega: X \longrightarrow \mathcal{O}$. We shall 
use the fact that any character extends to an algebra homomorphism $\Omega: \Lambda \longrightarrow \mathcal{O}$ via the specialization map 
$\lambda \mapsto \rho(\lambda) = \int_X \Omega(\sigma)d\lambda(\sigma)$. 

\begin{lemma}\label{trivial} We have that $\Omega(\lambda^*) = \Omega^{-1}(\lambda)$ for any element $\lambda \in X$. \end{lemma}

\begin{proof} Since $\Omega: X \longrightarrow \mathcal{O}$ is a group homomorphism, $\Omega(\sigma^{-1}) = \Omega^{-1}(\sigma)$. Hence,
\begin{align*} \Omega(\lambda^*) &= \int_X \Omega(\sigma) d\lambda^*(\sigma)= \int_X \Omega(\sigma^{-1}) d\lambda(\sigma) = 
\int_X \Omega^{-1}(\sigma)d\lambda(\sigma) = \Omega^{-1}(\lambda).\end{align*} \end{proof}

\begin{lemma}\label{nvj} In our constructions above of the measures $d\vartheta_{\varphi}^{(\delta)}(x)$ giving rise to the elements $\theta_{\varphi}^{(\delta)}(x)$ of 
$\mathcal{O}[[X]]$, our choice of sequence $(\varphi_n)_{n \geq 1} = (x_n \cdot \varphi)_{n \geq 1}$ in either case on $\delta$ defines a sequence of test vectors in the 
sense of \cite[$\S$ 1.2.3]{FMP}. \end{lemma}

\begin{proof} Fix $n \geq 1$. It suffices to check invariance properties at $\pp$ for each vector $\varphi_n = x_n \cdot \varphi \in \pi'$ for each choice of $\delta \in \lbrace 0, 1 \rbrace$. 

Let us first assume that $\delta = 0$. Recall that we fix a class $x_n \in B_{\pp}^{\times}/R_{\pp}^{\times}$ which is 
invariant by the action of $Z_n^{\times}$, where $Z_n = \mathcal{O}_{F_{\pp}} + \varpi_{\pp}^n \mathcal{O}_{K_{\pp}}$. 
We deduce that the local vector $x_n \cdot \varphi_{\pp} \in \pi_{\pp}$ is fixed by the action of $R_{\pp, n}^{\times}$, 
where $R_{\pp, n} \subset B_{\pp}$ is the unique (up to $K_{\pp}^{\times}$-conjugacy) order for which $c(R_{\pp, n})
= c(\Omega_{\pp}) = n$ and $d(R_{\pp, n}) = c(\pi_{\pp}) = \delta = 0$, whence the required local invariance follows 
from the work of Gross-Prasad \cite{GP}, cf. \cite[$\S 1.2.3$]{FMP} or the summary given above.

Let us assume now that $\delta = 1$. Recall that we fix a class $x_n \in B_{\pp}^{\times}/R_{\pp}^{\times}$ 
which corresponds under our fixed bijection $(\ref{b1})$ to a $1$-lattice $(L(0), L(1))$ we have by Lemma 
\ref{p1} (i) that $l_{\pp}(x_n) = \max(l_{\pp}(L(0)), l_{\pp}(L(1))) = n$, equivalently that $L(0)$ say is fixed by
$Z_{n-1}^{\times}$ and $L(1)$ by $Z_n^{\times}$ (as we can assume without loss of generality). We then 
deduce that the local vector $x_n \cdot \varphi_{\pp} \in \pi_{\pp}$ is fixed by the action of 
$R_{\pp, n}^{\times}$, where $R_{\pp, n}$ is the unique order of $B_{\pp}$ up to $K_{\pp}^{\times}$-conjugacy
which can be expressed as the intersection of maximal order $R_{\pp, n} = R_{1, \pp, n} \cap R_{2, \pp, n}$, 
where $c(R_{1, \pp, n}) = c(\Omega_{\pp}) = n$ and $c(R_{2, \pp, n}) = c(\Omega_{\pp}) - \delta = n-1$. The
required local invariance is then deduced from \cite{FMP}, again cf. \cite[$\S 1.2.3$]{FMP} or the summary 
above. \end{proof} 

Equipped with this result, we now give the following main interpolation formula. 
Recall that for a given an integer $n \geq 1$, we write $m(\mathcal{O}_{\pp^n})$ to denote the volume of 
$\widehat{\mathcal{O}}_{\pp^n}^{\times}$ in the space $K^{\times} \backslash {\bf{A}}_{K}^{\times} / {\bf{A}}_{F}^{\times}$ with respect to our fixed choice 
of Haar measure (which assigns $K^{\times} \backslash {\bf{A}}_K^{\times}/ {\bf{A}}_F^{\times}$ volume one). Recall as well that we write the class number 
of $F$ as $h(\mathcal{O}_F) = \vert \Pic(\mathcal{O}_F) \vert = \vert \widehat{F}^{\times} / F^{\times} \widehat{\mathcal{O}}_F^{\times} \vert$. 

\begin{proposition}\label{interpolation} 

Let $\Omega$ be a primitive character of $X$ of conductor $\pp^n$ with $n \geq 1$. 
Then, for each choice $\delta \in \lbrace 0, 1 \rbrace$, we have the following interpolation formula 
for the associated imprimitive $p$-adic $L$-function $L_{\pp}^{(\delta)}(\varphi, K)$, 

\begin{align}\label{IF} \Omega(L_{\pp}^{(\delta)}(\varphi, K)) 
&= \alpha_{\pp}^{2(\delta - 1 - n)} \cdot \left( \frac{h(\mathcal{O}_F)}{m(\mathcal{O}_{\pp^n})}\right)^2
\cdot (\varphi, \varphi) \cdot \mathcal{L}(1/2, \pi \times \Omega). \end{align} 
Here, $\mathcal{L}(1/2, \pi \times \Omega)$ denotes the algebraic $L$-value of Corollary \ref{alg+1} above, which we view as an element of 
$\overline{\bf{Q}}_p$ via our fixed embedding $\overline{\bf{Q}} \longrightarrow \overline{\bf{Q}}_p$. \end{proposition} \begin{proof} 

Lemma \ref{trivial} implies that for either choice of $\delta \in \lbrace 0, 1 \rbrace$, we have 
\begin{align*} \Omega(L_{\pp}^{(\delta)}(\varphi, K)) &= \Omega(\theta_{\varphi}^{(\delta)}( x)) \Omega({\theta_{\varphi}^{(\delta)}(x)}^*) 
= \Omega(\theta_{\varphi}^{(\delta)}(x)) \Omega^{-1}(\theta_{\varphi}^{(\delta)}(x)). \end{align*} 

Let us first assume that $\delta = 1$, whence we are reduced to evaluating the sums 
\begin{align*}\Omega(\theta_{\varphi}^{(1)}(x)) &= \int_X \Omega(\sigma) \theta_{\varphi}^{(1)}(x)(\sigma) = \alpha_{\pp}^{-n}  
\sum_{A \in X_n} \Omega(A) \varphi(A \star x_n). \end{align*} 
Here, the second equality follows by the definition of $\vartheta_n = \vartheta_n (\varphi, x_n)$ of $d\vartheta(\varphi, x)$. 
Since we know by Lemma \ref{nvj} that the vectors defined by $\varphi_n = x_n \cdot \varphi$ 
are test vectors in the sense of \cite[$\S 7.1$]{FMP}, we argue (cf.~\cite[(5)]{VaMSRI}) that we have the relation
\begin{align*}\Omega ( \theta_{\varphi}^{(1)}(x) )&= \frac{h(\mathcal{O}_F)}{m(\mathcal{O}_{\pp^n})} \cdot P_{\Omega}^B(\varphi).\end{align*} 
That is, we first argue that $n$ that the natural map $\Pic(\mathcal{O}_F) \rightarrow \Pic(\mathcal{O}_{\pp^n})$ is injective. 
Since $\Omega$ factors through 
$\Pic(\mathcal{O}_{\pp^n})/ \Pic(\mathcal{O}_F) \approx \widehat{K}^{\times}/\widehat{F}^{\times} K^{\times} \widehat{\mathcal{O}}_{\pp^n}^{\times}$ 
by our assumptions, and since it is invariant on $\widehat{\mathcal{O}}_{\pp^n}^{\times}$, we can express the period integral $P_{\Omega}^B(\varphi)$ as 
\begin{align*} P_{\Omega}^B(\varphi) &:= \int_{{\bf{A}}_K^{\times}/ {\bf{A}}_F^{\times}K^{\times}} \varphi(t) \Omega(t) dt 
= m(\mathcal{O}_{\pp^n}) \sum_{t \in \widehat{K}^{\times}/\widehat{F}^{\times}\widehat{\mathcal{O}}_{\pp^n}^{\times}} \varphi(t) \Omega(t). \end{align*}
On the other hand, we can decompose our finite sum over $X_n = \Pic(\mathcal{O}_{\pp^n})$ as 
\begin{align*} \sum_{A \in X_n} \Omega(A) \varphi(A) 
&= \sum_{\tau \in \Pic(\mathcal{O}_F)} \sum_{t \in \Pic(\mathcal{O}_{\pp^n})/\Pic(\mathcal{O}_F)} \Omega(t) \varphi(\tau t),\end{align*} 
using that $\Omega$ is trivial on $\Pic(\mathcal{O}_F)$. Now, observe that the inner sum here equals
\begin{align*} \sum_{\tau^{-1} t \in \Pic(\mathcal{O}_{\pp^n})/\Pic(\mathcal{O}_F)} \Omega(\tau^{-1} t) \varphi(t) 
&= \sum_{t \in \Pic(\mathcal{O}_{\pp^n})/\Pic(\mathcal{O}_F)} \Omega(t) \varphi(t), \end{align*} and hence that
\begin{align*} \sum_{A \in X_n} \Omega(A) \varphi(A) 
&= h(\mathcal{O}_F) \sum_{t \in \Pic(\mathcal{O}_{\pp^n})/\Pic(\mathcal{O}_F)} \Omega(t) \varphi(t)
= \frac{h(\mathcal{O}_F)}{m(\mathcal{O}_{\pp^n})} \cdot P_{\Omega}^B(\varphi). \end{align*}
The stated interpolation formula is then easy to deduce from Theorem \ref{formula+1}.

Let us now assume that $\delta = 0$, whence we are reduced to evaluating the sums 
\begin{align*} \Omega (\theta_{\varphi}^{(0)}(x)) &= \alpha_{\pp}^{1-n}
\sum_{A \in X_n} \Omega(A) \varphi (A \star x_n) - \alpha_{\pp}^{-n}\sum_{A \in X_n} \Omega(A)
\varphi(A \star x_{n-1}). \end{align*} Since the vectors $\varphi_n = x_n \cdot \varphi$ are again test
vectors in the sense of \cite[$\S 7.1$]{FMP} by Lemma \ref{nvj} above, we see that the second 
sum in this term must vanish, i.e. since the conductor of the test vector $\varphi_{n-1} = x_{n-1} 
\cdot \varphi$ appearing in this term does not have the same conductor as the character 
$\Omega$. Indeed, observe that given any $u \in {\bf{A}}_K^{\times}$, which we embed into the finite set 
$B^{\times} \backslash \widehat{B}^{\times}/ \widehat{R}^{\times}$ via our choice of (optimal) embedding $K \rightarrow D$, 
we must have that  $P_{\Omega} (u \cdot \varphi) = \Omega^{-1}(u)P_{\Omega}(\varphi)$, whence 
\begin{align*} \Omega (\theta_{\varphi}^{(0)}(x)) &= \alpha_{\pp}^{1-n} \sum_{A \in X_n} \Omega(A) \varphi (A \star x_n), \end{align*} 
which reduces us to the same style of proof as given for the $\delta = 1$ case above. \end{proof} 

Although it is clear, let us state for the record the following direct consequence.

\begin{corollary}\label{trivial2}

The specialization value $\Omega(L_{\pp}^{(\delta)}(\varphi, K))$ vanishes if and only if the complex central value $L(1/2, \pi \times \Omega)$ vanishes 
when the conductor of $\Omega$ is nontrivial. \end{corollary}

It is also clear that dividing out by $(\varphi, \varphi)$ gives an interpolation formula which does not depend on the scaling factor of $\varphi \in \pi'$. We record this as follows.

\begin{definition}
Let us in either case on $\delta \in \lbrace 0, 1 \rbrace$ define the {\it{primitive $p$-adic $L$-function}}  $\LLL_{\pp}(\pi', K) = \LLL_{\pp}^{(\delta)}{\pp}(\pi', K)$ 
associated to $\pi'$ and $X$ to be the element of $\Lambda = \mathcal{O}[[X]]$ obtained by dividing out by the inner product $(\varphi, \varphi)$, i.e. 
\begin{align}\label{pplfn} \LLL_{\pp}^{\delta}(\pi', K) &= L_{\pp}^{(\delta)}(\varphi, K)/(\varphi, \varphi).
\end{align} \end{definition} 

The following result is then easy to see from the discussion above.

\begin{corollary}\label{ninterpolation} The element $\LLL_{\pp}^{\delta}(\pi', K) \in \Lambda = \mathcal{O}[[X]]$ does not depend on the choice of scaling factor of 
$\varphi \in \pi'$, and satisfies the following interpolation property: For each primitive character $\Omega$ of $X$ of conductor $\pp^n$ with $n \geq 1$, 
we have that $\Omega(\LLL_{\pp}^{\delta}(\pi', K) ) =  \alpha_{\pp}^{2(\delta - 1 - n)} \cdot \left( h(\mathcal{O}_F)/m(\mathcal{O}_{\pp^n})\right)^2 \cdot
\mathcal{L}(1/2, \pi \times \Omega)$ in $\overline{\QQ}_p$. \end{corollary}

\subsection{The case of $k=1$} 

We now extend the construction given above to give a $p$-adic interpolation series for the setting of $k=1$, 
i.e. where $\epsilon = \epsilon(1/2, \pi \times \Omega)$ is generically equal to  $-1$. 

Keep all of the setup leading to Theorem \ref{YZZ} above, so that $\BB$ is the indefinite quaternion algebra which is ramified at all real places 
of $F$ except for one fixed place $\tau$, as well as the finite places corresponding to prime divisors of the inert level $\NN^{-}$. 
By Hypothesis \ref{pi}, we know that $\BB$ is split at $\pp$. Hence, we can and do fix an isomorphism $\BB_{\pp}^{\times} \approx \GL(F_{\pp})$.
Let $M = \lbrace M_U \rbrace_U$ be the Shimura curve over $F$ associated $\BB$. 
Let $A$ be an abelian variety defined over $F$ which is parametrized by $M$. 
Recall that according to \cite[$\S 3.2.2$]{YZZ}, we have an an automorphic representation 
$\pi_A = \Hom^0(J, A) = \varinjlim_H \Hom^0(J_H, A) =  \varinjlim_H \Hom_F(J_H, A) \otimes_{\ZZ} {\QQ}$ of $\BB^{\times}(\AF)$ defined over $\QQ$, 
and moreover that this representation decomposes as a product $\pi_A = \otimes_v \pi_{A, v}$ of absolutely irreducible representations $\pi_{A, v}$ of $\BB_v^{\times}$. 
Recall as well that for any choice of vector $\varphi \in \pi_A$ and CM point $Q \in M^{K^{\times}}(K^{\ab})$, we have that $\varphi(Q)$ is a well-defined element 
of the Mordell-Weil group $A(K^{\ab})$, where $K^{\ab}$ denotes the maximal abelian extension of $K$. 
Let us define $\mathcal{O}$ in this setting to be the tensor product $A(K^{\ab}) \otimes_{\ZZ} {\ZZ}_p$. 
We shall assume from now on that $\pi_A$ has trivial central character, hence that $\pi_A = \pi_A^{\vee}$ is self-dual. 
Thus $A$ is principally polarized, and we can fix an identification $A \approx A^{\vee}$. Granted Hypotheses \ref{local+1} and $\ref{pi}$ 
hold for $\pi = \pi_A$, the construction of the $p$-adic $L$-function $L_{\pp}(\pi, K)$ given above for $k=0$ extends formally, as we now explain. 

\subsubsection{The case of $\delta = 0$}

Assume again that the local order $R_{\pp} \subset \BB_{\pp}$ is Eichler of level $\pp^{\delta} = 1$, i.e. that $R_{\pp}$ is maximal. 
Recall again that we fix $V$ a simple left $\BB_{\pp}$-module for which $V \approx F_{\pp}^2$ as an $F_{\pp}$-vector space, 
that we let $\LL(V)$ denote the set of $\OFP$-lattices in $V$, and that we have the bijection $(\ref{d0})$.
The arguments of Lemma \ref{sc0}, Corollary \ref{6.5C}, and Proposition \ref{DR0} above imply the following result. 

\begin{corollary}\label{DR0-1} 

Let $\varphi \in \pi_A$ be a nonzero $\pp$-ordinary decomposable vector whose component at $\pp$ is fixed by $R_{\pp}^{\times}$. 
Let $(x_n)_{n \geq 0}$ denote the sequence of classes in $\BB_{\pp}^{\times} /R_{\pp}^{\times}$ as defined above. 
Writing $T_{\pp}^l$ to denote the lattice Hecke operator defined above, 
let $a_{\pp} = a_{\pp}(\pi_A)$ denote the eigenvalue of $T_{\pp}^l$ acting on $\varphi$, with $\alpha_{\pp} = \alpha_{\pp}(\pi_A)$ the unit root of the Hecke polynomial 
$t^2 - a_{\pp}t + q_{\pp}$. Then, the system of mappings $\lbrace \phi_n \rbrace_{n \geq 1} = \lbrace \phi_n(x_n, \varphi) \rbrace_{n \geq 1}$ defined by 
\begin{align}\label{d0-1} \phi_n: X_n &\longrightarrow \mathcal{O}, 
~~~ A \longmapsto \alpha_{\pp}^{1-n} \cdot \varphi(A \star x_n) - \alpha_{\pp}^{-n} \cdot 
\varphi(A \star x_{n-1}) \end{align} determines a distribution on the profinite group $X = \varprojlim X_n$. \end{corollary}

\subsubsection{The case of $\delta = 1$}

Let us now assume that the order $R_{\pp} \subset \BB_{\pp}$ is Eichler of level $\pp^{\delta} = \pp$. 
Recall that in this case, we write $\LL_1 = \LL_1(V)$ to denote the set of $1$-lattices of $\LL = \LL(V)$. 
We then fix a $1$-lattice $L_0 = (L_0(0), L_0(1))$ whose stabilizer under the transitive action of $\BB_{\pp}^{\times}$ is equal to $R_{\pp}^{\times}$, 
which allows us to define the bijection $(\ref{d1})$. Again, we fix a $K_{\pp}$-basis $e$ of $V$, and for each integer $n \geq 1$ write $L_n = L_n(e)$
to denote the lattice in $\LL = \LL(V)$ defined by $L_n = \pp^{- \lfloor \frac{n}{2} \rfloor/ 2} Z_n e$. 
For each $n \geq 1$, let $J_n = J_n(e)$ to denote the $1$-lattice in $\mathcal{L}_1 = \mathcal{L}_1(V)$ defined by 
\begin{align*} J_n &= \begin{cases} (L_{n-1}, L_n) &\text{ if $n \equiv 0 \mod 2$} \\ (L_n, L_{n-1}) &\text{ if $n \equiv 1 \mod 2$}\end{cases} \end{align*} 
Let $x_n = x_n(e) \in \BB_{\pp}^{\times}/R_{\pp}^{\times}$ be the class corresponding to $J_n$ under $(\ref{d1})$. 
The arguments of Lemma \ref{p1}, Corollary \ref{r1}, and Proposition \ref{DR1} imply the following result.

\begin{corollary}\label{DR1-1} 

Let $\varphi \in \pi_A$ be a nonzero decomposable vector whose component at $\pp$ is fixed by $R_{\pp}^{\times}$. 
Let $(x_n)_{n \geq 0}$ denote the sequence of classes in $\BB_{\pp}^{\times} /R_{\pp}^{\times}$ defined by the $1$-lattices $J_n$ above. 
Suppose that $\varphi$ is a $\pp$-ordinary eigenvector for both of the operators $T_{\pp}^l$ and $T_{\pp}^{u}$
with common (unit) eigenvalue $\alpha_{\pp} = \alpha_{\pp}(\pi_A)$ Then, the system $\lbrace \phi_n \rbrace_{n \geq 1} = \lbrace \phi_n(x_n, \varphi) \rbrace_{n \geq 1}$ 
of mappings \begin{align}\label{d1-1} \phi_n: X_n &\longrightarrow \mathcal{O}, ~~~ A \longmapsto \alpha_{\pp}^{-n} \cdot \varphi(A \star x_n)_{n \geq 1} \end{align} 
defines a distribution on the profinite group $X = \varprojlim X_n$. \end{corollary}

\subsubsection{Definition of the interpolation series}

We have now defined in either case on the exponent $\delta \in \lbrace 0, 1 \rbrace$ an $\mathcal{O}$-valued distribution 
\begin{align*} d \phi^{(\delta)} (\varphi, x(e)) &= \lbrace \phi_n^{(\delta)}(\varphi, x_n(e)) \rbrace_{n \geq 1} \end{align*} on 
$X$ whose corresponding element in $\mathcal{O} [[X]] = \varprojlim_n \mathcal{O}[X_n]$ 
we denote by \begin{align*} \Phi_{\varphi}^{(\delta)} = \Phi_{\varphi}^{(\delta)}(x(e)) = (\Phi_{\varphi}^{(\delta)}(x_n(e))). \end{align*} 
These elements depend on the choice of $K_{\pp}$-basis $e$ of $V$ in the same way as described for Lemma \ref{choiceofe}. We 
again account for this ambiguity by considering the elements
\begin{align*} \Phi_{\varphi}^{(\delta)}(x(e)) \Phi_{\varphi}^{(\delta)} (x(e))^* \text{~~ in ~$\mathcal{O}[[X]]$}, \end{align*} 
where $\Phi_{\varphi}^{\delta} (x(e))^*$ denotes the image of $\Phi_{\varphi}^{\delta} (x(e))$ under the involution of $\mathcal{O}[[X]]$ 
induced by sending elements $\sigma \in X$ to their inverses $\sigma^{-1} \in X$.

\begin{definition}

Let $\varphi \in \pi_A$ be a nonzero decomposable vector satisfying Hypotheses \ref{pi} and \ref{local+1} above. 
Let $e$ be any $K_{\pp}$-basis of $V$. Let $x(e) = (x_n(e))_{n \geq \delta}$ be either the sequence of classes 
defined above. Let us now define the well-defined elements \begin{align}\label{intelt-1}  D_{\pp}^{(\delta)}(\varphi, K) 
&= \Phi_{\varphi}^{(\delta)}(x(e)) \Phi_{\varphi}^{(\delta)}  (x(e))^* \in  \mathcal{O} [[X]].\end{align} \end{definition}

\subsubsection{Interpolation}

The elements defined in $(\ref{intelt-1})$ satisfy the following interpolation property. Recall that we write $\langle \cdot, \cdot \rangle$ 
to denote the $L$-linear N\'eron-Tate pairing \begin{align*} \langle \cdot, \cdot \rangle_L: A(\overline{F})_{\QQ} \otimes_R
A(\overline{F})_{\QQ} &\longrightarrow L \otimes_{\QQ} {\CC} \end{align*} defined in \cite[$\S 1$]{YZZ}. 
Note that this construction extends in a natural way to define an $\mathcal{O}_L$-linear N\'eron-Tate height pairing 
\begin{align*} \langle \cdot, \cdot \rangle_{\mathcal{O}_L}: A(\overline{F})_{\QQ} \otimes_R A(\overline{F})_{\QQ} 
&\longrightarrow \mathcal{O}_L \otimes_{\QQ} {\CC}. \end{align*} 
We argue that these pairings also extend in a natural way by ${\ZZ}_p$-linearity to define pairings on the elements in 
$\mathcal{O}$ we have considered above, and hence that we can derive the following straightforward consequence from the discussion above.

\begin{proposition}\label{interpolation-1} 

Let $\Omega$ be a primitive ring class character of $X$ of conductor $\pp^n$ with $n \geq 1$. Then, for each choice of exponent $\delta \in \lbrace 0, 1 \rbrace$, 
we have the following interpolation formula for the elements $D_{\pp}^{(\delta)}(\varphi, K)$ defined in $(\ref{intelt-1})$ above composed with
the $\mathcal{O}_L$-linear N\'eron-Tate height pairing $\langle ~,~\rangle_L$,  
\begin{align*}  \langle ~,~\rangle_L &\circ \Omega(D_{\pp}^{(\delta)}(\varphi, K)) \\ 
&= \alpha_{\pp}(\pi_A)^{-2(1-\delta - n)} \cdot \left( \frac{h(\mathcal{O}_F)}{m(\mathcal{O}_{\pp^n})}\right)^2 \cdot
 \frac{\zeta(2) L'(1/2, \pi_A \times \Omega)}{4 L(1, \eta)^2 L(1, \pi_A, \ad)}
\cdot  \alpha (x_n \cdot \varphi, x_n \cdot \varphi) . \end{align*} \end{proposition} 

\begin{proof} The proof is formally identical to that of Proposition \ref{interpolation} above, 
using Theorem \ref{YZZ} in lieu of Theorem \ref{formula+1} for the interpolation values. \end{proof}

\begin{remark}[Acknowledgements.] I should like to thank Christophe Cornut first and foremost for extremely helpful comments about the main construction. 
I should also like to thank Masataka Chida, Henri Darmon, Daniel Disegni, Fred Diamond, Toby Gee, Min-Lun Hsieh, Kimball Martin, Udi de Shalit, Xinxi Yuan, 
Shouwu Zhang, Wei Zhang, and an anonymous referee for helpful comments and exchanges. \end{remark}

\end{document}